\documentclass[11pt]{amsart}
\usepackage{amsmath}
\usepackage{amsxtra}
\usepackage{amscd}
\usepackage{amsthm}
\usepackage{amsfonts}
\usepackage{amssymb}
\usepackage{eucal}

\newcommand{\ket}[1]{{| #1 \rangle}}      
\newcommand{\br}[1]{{\langle #1 \rangle}}  

\newcommand{\ch}{{\rm ch}}

\newcommand{\bR}{\bar{R}}

\newcommand{\bea}{\begin{eqnarray}}
\newcommand{\ena}{\end{eqnarray}}
\newcommand{\be}{\begin{eqnarray*}}
\newcommand{\en}{\end{eqnarray*}}
\def\bel{\begin{eqnarray}}
\def\enl{\end{eqnarray}}

\newcommand{\C}{{\mathbb C}}

\newcommand{\Z}{{\mathbb Z}}
\newcommand{\al}{{\alpha}}

\newcommand{\la}{{\lambda}}

\newcommand{\mc}{\mathcal}

\def\P{\mathcal P}

\numberwithin{equation}{section}

\numberwithin{equation}{section}
\newtheorem{thm}{Theorem}[section]
\newtheorem{prop}[thm]{Proposition}
\newtheorem{lem}[thm]{Lemma}
\newtheorem{cor}[thm]{Corollary}

\theoremstyle{remark}
\newtheorem{rem}[thm]{Remark}

\newtheorem{definition}[thm]{Definition}
\newtheorem{dfn}[thm]{Definition}

\newcommand{\ah}{\widehat{\mathfrak a}}

\newcommand{\h}{\mathfrak{h}}
\newcommand{\na}{\mathfrak{n}}
\newcommand{\ve}{\varepsilon}

\newcommand{\n}{\mathfrak{n}}

\newcommand{\T}{\otimes}

\newcommand{\nh}{\widehat{\mathfrak{n}}}

\newcommand{\A}{\nh}
\renewcommand{\deg}{\mathop{\rm deg}}

\newcommand{\Uv}{U_v(\mathfrak{sl}_3)}
\newcommand{\Uvv}{U_{v^{-1}}(\mathfrak{sl}_3)}

\newcommand{\V}{\mathcal V}

\begin{document}

\title[Principal subspaces and quantum Toda Hamiltonian]
{Principal $\widehat{\mathfrak{sl}_3}$ subspaces and quantum Toda Hamiltonian}
\author{B. Feigin, E. Feigin, M. Jimbo, T. Miwa and E. Mukhin}
\address{BF: Landau institute for Theoretical Physics,
Chernogolovka,
142432, Russia and \newline
Independent University of Moscow, Russia, Moscow, 119002,
Bol'shoi Vlas'evski per., 11}
\email{feigin@mccme.ru}
\address{EF:
Tamm Theory Division, Lebedev Physics Institute, Russia, Moscow, 119991,
Leninski pr., 53 and\newline
Independent University of Moscow, Russia, Moscow, 119002,
Bol'shoi Vlas'evski per., 11}
\email{evgfeig@gmail.com}
\address{MJ: Graduate School of Mathematical Sciences,
The University of Tokyo, Tokyo 153-8914, Japan}
\email{jimbomic@ms.u-tokyo.ac.jp}
\address{TM: Department of Mathematics,
Graduate School of Science,
Kyoto University, Kyoto 606-8502,
Japan}\email{tetsuji@math.kyoto-u.ac.jp}
\address{EM: Department of Mathematics,
Indiana University-Purdue University-Indianapolis,
402 N.Blackford St., LD 270,
Indianapolis, IN 46202}\email{mukhin@math.iupui.edu}


\begin{abstract}
We study a class of representations of the Lie algebra
$\na\T\C[t,t^{-1}]$, where $\na$ is a nilpotent subalgebra of
$\mathfrak{sl}_3$. We derive Weyl-type (bosonic) character formulas
for these representations. We establish a connection between the
bosonic formulas and the Whittaker vector in the Verma module for
the quantum group $U_v(\mathfrak{sl}_3)$. We also obtain a fermionic
formula for an eigenfunction of the $\mathfrak{sl}_3$ quantum Toda
Hamiltonian.
\end{abstract}
\maketitle


\section{Introduction}
Let $\na=\C e_{21}\oplus \C e_{32}\oplus\C e_{31}$ be the nilpotent subalgebra of
the complex Lie algebra $\mathfrak{sl}_3$,
and let $\A=\na\T\C[t,t^{-1}]$ be the corresponding current algebra.
In this paper we study a class of $\A$-modules.
The simplest example of the modules in question is the
principal subspace $V^k$ of the level $k$ vacuum  representation of
$\widehat{\mathfrak{sl}_3}$ (see \cite{FS}). Namely, let $M^k$ be the
level $k$ vacuum representation of the affine Lie algebra $\widehat{\mathfrak{sl}_3}$.
{}Fix a highest weight vector $v^k\in M^k$. Then
$$V^k=U(\A)\cdot v^k.$$
The principal subspaces are studied in \cite{AKS}, \cite{C}, \cite{CLM}, \cite{FS}, \cite{G},
\cite{LP}, \cite{P}.
In particular, the following fermionic formula is available for the
character of $V^k$:
\begin{equation}\label{fV}
\mathrm{ch} V^k=
\sum_{\genfrac{}{}{0pt}{}{n_1,\dots,n_k\ge 0}{m_1,\dots,m_k\ge 0}}
\frac{z_1^{\sum_{i=1}^k in_i}z_2^{\sum_{i=1}^k im_i}
q^{\sum_{i,j=1}^k \min(i,j) (n_in_j - m_in_j + m_im_j)}}
{(q)_{n_1}\dots (q)_{n_k} (q)_{m_1}\dots (q)_{m_k}},
\end{equation}
where $(a)_n=\prod_{i=0}^{n-1} (1-aq^i)$.
One of our results
is the following new formula for $\ch V^k$.
{}For non-negative integers $d_1$, $d_2$, set
\be
&&I_{d_1,d_2}(z_1,z_2)=
\frac{
 (q z_1^{-1} z_2^{-1})_{d_1+d_2}}
{(q)_{d_1} (q)_{d_2}
(qz_1^{-1})_{d_1} (qz_2^{-1})_{d_2}
(q z_1^{-1} z_2^{-1})_{d_1}(q z_1^{-1} z_2^{-1})_{d_2}},
\\
&&J_{d_1,d_2}(z_1,z_2)
=\frac{I_{d_1,d_2}(z_1,z_2)}{(qz_1)_\infty (qz_2)_\infty (qz_1z_2)_\infty}.
\en
We show that
\begin{equation}
\label{VK}
\ch V^k= \sum_{d_1,d_2\ge 0}
z_1^{kd_1} z_2^{kd_2} q^{k(d_1^2+d_2^2-d_1d_2)}
J_{d_1,d_2}(z_1q^{2d_1-d_2}, z_2q^{2d_2-d_1}).
\end{equation}
In the right hand side, each summand is understood as 
a power series expansion in $z_1,z_2$. 
{}Formula (\ref{VK}) was conjectured in \cite{FS1}.

The functions $I_{d_1,d_2}(z_1,z_2)$ are known to be the coefficients of the expansion of
an eigenfunction of  the $\mathfrak{sl}_3$ quantum Toda Hamiltonian (see \cite{GL}).
Namely, define the generating function
$$I(Q_1,Q_2,z_1,z_2)=\sum_{d_1,d_2\ge 0} Q_1^{d_1}Q_2^{d_2} I_{d_1,d_2}(z_1,z_2).$$
The $\mathfrak{sl}_3$
quantum Toda Hamiltonian $\widehat H$ is an operator of the form
\be
\widehat H=q^{\partial/\partial t_0}+ q^{\partial/\partial t_1}(1-Q_1)+
q^{\partial/\partial t_2}(1-Q_2)
\en
acting on the space of functions in variables $Q_1$, $Q_2$.
The variables $t_i$ are introduced by $Q_i=e^{t_{i-1}-t_i}$
and
$$q^{\partial/\partial t_j}: t_i\mapsto t_i+\delta_{ij} \mathrm{ln} q.$$
Let the variables $p_1$, $p_2$ be such that $z_1=p_1^{-2}p_2$, $z_2=p_1p_2^{-2}$.
Then
\begin{equation}\label{Toda}
\widehat{H}  \left(p_1^{\frac{t_0-t_1}{\ln q}}
p_2^{\frac{t_1-t_2}{\ln q}}I(Q_1,Q_2,z_1,z_2)\right)=(p_1+p_1^{-1}p_2+p_2^{-1})I(Q_1,Q_2,z_1,z_2)
\end{equation}
(see \cite{GL}, \cite{E}, \cite{BF}).
Equation (\ref{Toda}) can be rewritten as a set of recurrent relations
\begin{multline}\label{Todarec}
(p_1 (q^{d_1}-1)+p_1^{-1}p_2 (q^{d_2-d_1}-1)+
p_2^{-1} (q^{-d_2}-1))I_{d_1,d_2}(z_1,z_2)
\\
=
p_2 p_1^{-1}q^{d_2-d_1}I_{d_1-1,d_2}(z_1,z_2)+
p_2^{-1} q^{-d_2}I_{d_1,d_2-1}(z_1,z_2).
\end{multline}
In this paper we call (\ref{Todarec}) the Toda recursion.

One of the consequences of the formulas (\ref{fV}) and (\ref{VK})
is the following recurrence relations for the rational functions
$I_{d_1,d_2}(z_1,z_2)$:
\begin{equation}
I_{d_1,d_2}(z_1,z_2)=\sum_{n_1=0}^{d_1} \sum_{n_2=0}^{d_2}
\frac{z_1^{-n_1}z_2^{-n_2}q^{n_1^2+n_2^2-n_1n_2}}{(q)_{d_1-n_1}(q)_{d_2-n_2}}
I_{n_1,n_2}(z_1,z_2),
\end{equation}
which leads to the fermionic formula
\begin{equation}\label{rectferm}
I_{d_1,d_2}(z_1,z_2)=
\sum_{\{n_i\}_{i>0}, \{m_i\}_{i>0}}
\frac{z_1^{-\sum_{i>0} n_i}z_2^{-\sum_{i>0} m_i}
q^{\sum_{i>0}(n_i^2+m_i^2-n_im_i)}}
{(q)_{d_1-n_1}(q)_{n_1-n_2}\dots (q)_{d_2-m_1}(q)_{m_1-m_2}\dots}\,,
\end{equation}
where the sum is over all sequences $\{n_i\}$, $\{m_i\}$ such that
$$
n_i,m_i\in\Z_{\ge 0},\quad d_1\ge n_1\ge n_2\ge\dots,\quad
d_2\ge m_1\ge m_2\ge \dots,
$$
and $n_i,m_i$ vanish for almost all $i$.
We conjecture that the obvious generalization of
(\ref{rectferm}) to the case of $\mathfrak{sl}_n$
gives the coefficients of an
eigenfunction for the corresponding quantum Toda Hamiltonian.

Let us briefly explain our approach to the computation of the
character formulas.
Recall (see \cite{K1}) the Weyl-Kac formula for the
character of an integrable
irreducible representation $M_\la$ of a Kac-Moody Lie algebra. It is written
as a sum over the set of extremal vectors in $M_\la$.
We call the summands the contributions of the extremal vectors.
There are two different ways to compute these contributions.
The first one is algebraic (see \cite{K1}, \cite{Kum})
and uses the BGG resolution.
The second uses
the realization of $M_\la$ as a dual space of sections of
a certain line bundle
on the generalized flag manifold and the Lefschetz fixed point formula (see \cite{Kum}).
We want to obtain a formula of the same structure by a
combinatorial method.
Let us explain our method on the example of $M_\la$.
In this case the extremal vectors are
labeled by elements of the affine Weyl group $W$,
and the character formula can be written as
\begin{equation}\label{WK}
\ch M_\la=\sum_{w\in W} \exp(w\la)\lim_{n\to\infty} (\exp(-w(n\la))\ch M_{n\la}).
\end{equation}
Roughly speaking, we compute
the character of $M_{n\la}$ in the ``vicinity" of the extremal vectors
and sum up the results.
We apply the same approach to the characters of $\A$-modules.
We use combinatorial tools to compute the terms corresponding to
$\lim_{n\to\infty} (\exp(-w(n\la))\ch M_{n\la})$ in (\ref{WK}).
Thus to obtain
a {\it bosonic}
formula for the character of an $\A$-module $M$ we follow the three steps:
\begin{itemize}
\item[(i)\ ]
find (guess) the set of extremal vectors of $M$;
\item[(ii)\ ]
find (guess) the contribution of each vector;
\item[(iii)\ ]
prove that the sum of all contributions equals to the character of $M$.
\end{itemize}
Step (i) is more or less easy, while steps (ii) and (iii) are subtler.
{}For example, for the principal subspace
$V^k$, the extremal vectors  are labeled by $\Z^2_{\ge 0}$
and the corresponding bosonic formula
is given by (\ref{VK}).
In particular, the contributions of the extremal vectors are
given by $J_{d_1,d_2}(z_1q^{2d_1-d_2}, z_2q^{2d_2-d_1})$.
In order to complete step (iii) for $V^k$,
we introduce a set of $\A$-modules which contains,
in particular, all principal subspaces in integrable
$\widehat{\mathfrak{sl}_3}$-modules. We describe these $\A$-modules below.

Let
$$e_{ij}[n]=e_{ij}\T t^n\in\A,\quad e_{ij}(z)=\sum_{n\in\Z} e_{ij}[n]\T z^{-n-1}.$$
The module $V^k$ can be described as a cyclic
$\A$-module with a cyclic vector $v$ such that 
the relations
$$
e_{21}(z)^{k+1}=0,\quad e_{32}(z)^{k+1}=0
$$
hold on $V^k$,
and that the cyclic vector satisfies
$$
e_{ij}[n]v=0
\quad
(n\ge 0).
$$
Let $k_1, k_2, l_1,l_2,l_3$ be non-negative
integers satisfying $k_1\le k_2$.
The module $U^{k_1,k_2}_{l_1,l_2,l_3}$
is a cyclic $\A$-module with a cyclic vector $v$ such that
the relations
$$e_{21}(z)^{k_1+1}=0,\ e_{32}(z)^{k_2+1}=0$$
hold on $U^{k_1,k_2}_{l_1,l_2,l_3}$, and that the cyclic vector satisfies
\be
&&e_{21}[n]v=0,\quad e_{32}[n]v=0
\quad (n>0),
\\
&&e_{21}[0]^{l_1+1}v=0,\quad e_{31}[1]^{l_3+1}v=0,\quad e_{32}[0]^{l_2+1}v=0.
\en
Similarly, the module $V^{k_1,k_2}_{l_1,l_2,l_3}$
is a cyclic $\A$-module with the cyclic vector $v$
such that
the relations
$$
e_{21}(z)^{k_1+1}=0,\ e_{32}(z)^{k_2+1}=0
$$
hold on $V^{k_1,k_2}_{l_1,l_2,l_3}$, and that the cyclic vector satisfies
\be
&&e_{21}[n]v=0,\quad e_{32}[n]v=0,\quad e_{31}[n]v=0
\quad (n>0),
\\
&&e_{21}[0]^{l_1+1}v=0,\quad e_{32}[0]^{l_2+1}v=0,
\\
&&e_{21}[0]^\alpha e_{31}[0]^{l_3+1-\alpha}v=0\quad (0\le \alpha\le l_3+1).
\en
The structure of the set
of extremal vectors is more complicated for
the modules $U^{k_1,k_2}_{l_1,l_2,l_3}$ and $V^{k_1,k_2}_{l_1,l_2,l_3}$
than for $V^k$.
The extremal vectors of the modules
$U^{k_1,k_2}_{l_1,l_2,l_3}$ and $V^{k_1,k_2}_{l_1,l_2,l_3}$
are labeled by $({\bf d},\sigma)$, where
${\bf d}\in\Z_{\ge 0}^3$ and $\sigma$ is an element of the Weyl group of
$\mathfrak{sl}_3$.
On the other hand,
the computation of the contributions of extremal vectors
for $U^{k_1,k_2}_{l_1,l_2,l_3}$
and  $V^{k_1,k_2}_{l_1,l_2,l_3}$ is simpler than for $V^k$.
We write these contributions  explicitly and
prove that they sum up
to the characters of the modules $U^{k_1,k_2}_{l_1,l_2,l_3}$
and $V^{k_1,k_2}_{l_1,l_2,l_3}$
when the parameters $l_1,l_2,l_3$ belong to a certain region.
To do that, we show that the characters
and the sums of the contributions
satisfy the same set of recurrent relations. We also show that
the solution of these recursion relations is unique.

The principal space $V^k$ is isomorphic to $U^{k,k}_{0,0,0}$.
Equating the  bosonic formula and (\ref{VK}), we arrive at the
identity
\begin{equation}\label{IJ}
I_{d_1,d_2}(z_1,z_2)=\sum_{n=0}^{\min(d_1,d_2)} I_{d_1,d_2,n}(z_1,z_2),
\end{equation}
where
\begin{multline*}
I_{d_1,d_2,n}(z_1,z_2)\\
=\frac{1}{(q)_{d_1-n}(q)_{d_2-n}(q)_n}
\frac{(qz_2)_\infty}{(qz_1^{-1})_{d_1-n}
(qz_1^{-1}z_2^{-1})_n (q^{d_1-2n+1}z_2)_\infty
(q^{-d_1+2n+1}z_2^{-1})_{d_2-n} (qz_2)_{d_1-n}
 (qz_2^{-1})_n}.
\end{multline*}
The functions $I_{d_1,d_2}(z_1,z_2)$ in the limit $q\to 1$
appear in the study of the Whittaker functions (see \cite{IS}).
{}For generic $q$, they are also closely related to the Whittaker vectors
in Verma modules of quantum groups
(see \cite{Kos}, \cite{E}, \cite{S}).
We interpret (\ref{IJ}) in terms of the representation theory.
Namely, let $\omega$ and $\bar\omega$ be the Whittaker vectors
in the Verma modules of the quantum groups
$U_{q^{1/2}}(\mathfrak{sl}_3)$ and
$U_{q^{-1/2}}(\mathfrak{sl}_3)$.
We fix the decompositions of $\omega$ and $\bar\omega$
in the Gelfand-Tsetlin bases:
\be
\omega=\sum_{d_1,d_2\ge 0} \sum_{n=0}^{\min(d_1,d_2)}
\omega_{d_1,d_2,n},\quad
\bar\omega=\sum_{d_1,d_2\ge 0} \sum_{n=0}^{\min(d_1,d_2)}
\bar\omega_{d_1,d_2,n}.
\en
We prove that
$$
I_{d_1,d_2,n}=(\omega_{d_1,d_2,n},\bar\omega_{d_1,d_2,n}),\quad
I_{d_1,d_2}=(\sum_{n=0}^{\min(d_1,d_2)} \omega_{d_1,d_2,n},
\sum_{n=0}^{\min(d_1,d_2)} \bar\omega_{d_1,d_2,n}),
$$
where $(~,~)$ denotes the dual pairing.
{}For the connection of the Whittaker vectors and Toda equations see also
\cite{GL}, \cite{BF}, \cite{GKLO}.

Our paper is organized as follows.

In Section 2, we recall known bosonic and fermionic formulas for the simplest
case of $\mathfrak{sl}_2$. This section is meant to be an illustration of
the discussions which follow.
In Section 3, we introduce the family of $\nh$-modules
\bea
\{U^{k_1,k_2}_{l_1,l_2,l_3} \mid (l_1,l_2,l_3)\in P^{k_1,k_2}_U\},
\quad
\{V^{k_1,k_2}_{l_1,l_2,l_3} \mid (l_1,l_2,l_3)\in P^{k_1,k_2}_V\},
\label{UV}
\ena
where the index sets
$P^{k_1,k_2}_U$, $P^{k_1,k_2}_V$ are
defined in the text (see \eqref{Pphi}, \eqref{Ppsi}).
Studying the structure of some of their subquotients, we derive a
recurrent upper estimate for their characters.
In Section 4, we introduce another family of $\nh$-modules
\bea
\{(U_{VO})^{k_1,k_2}_{l_1,l_2,l_3} \mid (l_1,l_2,l_3)\in
\bR^{k_1,k_2}_U\},
\quad
\{(V_{VO})^{k_1,k_2}_{l_1,l_2,l_3} \mid (l_1,l_2,l_3)\in R^{k_1,k_2}_V\}
\label{UVVO}
\ena
using the vertex operator construction. These modules are parametrized by
subsets $\bR^{k_1,k_2}_U\subset P^{k_1,k_2}_U$,
$R^{k_1,k_2}_V\subset P^{k_1,k_2}_V$
(see \eqref{Rphib},\eqref{Rpsi}),
and are quotients of the corresponding modules \eqref{UV}.
{}For these modules we derive a recurrent lower estimate for the characters.
In Section 5, we show
the uniqueness of solutions for the recurrent estimates
(see Proposition \ref{loop} for the precise statement),
and prove that in the parameter regions
$\bR^{k_1,k_2}_U$ and $R^{k_1,k_2}_V$,
the modules \eqref{UV} and \eqref{UVVO} are isomorphic
(Theorem \ref{thm:equalities}).
In Section 6, we proceed to write bosonic formulas for these modules
utilizing an inductive structure with respect to
the rank of the algebra.
We start by recalling previous results on bosonic formulas for modules over
the abelian subalgebra $\widehat{\mathfrak{a}}\subset\nh$ spanned by
$e_{21}[n]$, $e_{31}[n]$ ($n\in\Z$).
To make distinction we call the latter modules
$\mathfrak{sl}_3$ {\it small principal subspaces}.
Combining the characters for the $\mathfrak{sl}_2$ principal subspaces and
those of the $\mathfrak{sl}_3$ small principal subspaces,
we present a family of formal series. Then we
prove that these formal series coincide with the characters of
\eqref{UV} in appropriate regions (Theorem \ref{chsp};
the region of validity
for $V^{k_1,k_2}_{l_1,l_2,l_3}$ is $\bR^{k_1,k_2}_V$, while
for $U^{k_1,k_2}_{l_1,l_2,l_3}$ it
is a subset $\tilde{R}^{k_1,k_2}_U\subset \bR^{k_1,k_2}_U$,
see \eqref{Rtilde}).
In Section 7, we consider the special case $k_1=k_2$. In this case
some of the terms in the bosonic formula
can be combined to give a simpler result
of the type discussed in \cite{GL}. We obtain such a formula
(Theorem \ref{g-l}), which includes as a particular case
the character formula
for the principal subspace $V^k$ (Corollary \ref{CorVI}).
In the final Section 8, we discuss the connection to Whittaker vectors
in the Verma modules of quantum groups.

Throughout the text, $e_{ab}$ denotes the matrix unit with $1$ at the
($a,b$)-th place and $0$ elsewhere.
{}For a graded vector space $M=\oplus_{m_1,\cdots,m_l,d\in\Z}M_{m_1,\cdots,m_l,d}$
with finite dimensional homogeneous components $M_{m_1,\cdots,m_l,d}$,
we call the formal Laurent series
\be
{\rm ch}_{z_1,\cdots,z_l,q} M=\sum_{m_1,\cdots,m_l,d\in\Z}
(\dim M_{m_1,\cdots,m_l,d})z_1^{m_1}\cdots z_l^{m_l}q^d
\en
the {\it character} of $M$.
In the text we deal with the case $l=1$ or $2$.
We often suppress $q$ from the notation as it does not change throughout
the paper.
{}For two formal series with integer coefficients
$$
f^{(i)}=\sum_{m_1,\cdots,m_l,d}f^{(i)}_{m_1,\cdots,m_l,d}z_1^{m_1}\cdots z_l^{m_l}q^d
\quad (i=1,2),
$$
we write $f^{(1)}\le f^{(2)}$ to mean
$f^{(1)}_{m_1,\cdots,m_l,d}\le f^{(2)}_{m_1,\cdots,m_l,d}$ for all $m_1,\cdots,m_l,d$.


\section{Bosonic formula for the case of $\widehat{\mathfrak{sl}}_2$.}
In this section, we study the characters of the principal subspaces
of integrable modules for $\widehat{\mathfrak{sl}}_2$.
We present fermionic and bosonic formulas for these characters.
The contributions of the extremal vectors are calculated in
two different ways,
one from the combinatorial set which labels a monomial basis, and another
from the fermionic formula.
\subsection{Principal spaces for $\widehat{\mathfrak{sl}}_2$}
Consider the Lie algebra
$\mathfrak{sl}_2=\C e_{12}\oplus\C e_{21}\oplus\C(e_{11}-e_{22})$
where $e_{ab}$ are the $2\times2$ matrix units. We also
consider the affine Lie algebra $\widehat{\mathfrak{sl}}_2$
spanned by the central element $c$ and
\be
e[n]=e_{12}\otimes t^n,\ f[n]=e_{21}\otimes t^n,\
h[n]=(e_{11}-e_{22})\otimes t^n
\ (n\in\Z).
\en
Let $M^k_l$ be the level $k$ irreducible highest weight module
of the affine Lie algebra $\widehat{\mathfrak{sl}}_2$
with the highest weight $l$ $(0\leq l\leq k)$.  
On $M^k_l$, $c$ acts as a scalar $k$.
The module  $M^k_l$ has a highest weight vector $v^k_l$
characterized up to scalar multiple by 
\be
&&x[n]v^k_l=0\quad(x=e,f,h;n>0),\\
&&e[0]v^k_l=0,\quad f[0]^{l+1}v^k_l=0,\quad h[0]v^k_l=lv^k_l\,.
\en
Let $\hat{\mathfrak n}$ be the subalgebra of $\widehat{\mathfrak{sl}}_2$
generated by $f[n]$ $(n\in\Z)$.
Let $V^k_l$ be the subspace of $M^k_l$
generated by $f[-n]$ $(n\geq0)$ from $v^k_l$.
The space $V^k_l$ is called the principal subspace of $M^k_l$.
It is known \cite{FS} that
$V^k_l$ is isomorphic to the cyclic module $U(\hat{\mathfrak n})v^k_l$
with the defining relations
\be
f(z)^{k+1}=0,\quad f[0]^{l+1}v^k_l=0, \quad f[n]v^k_l=0\ (n>0),
\en
where $f(z)=\sum_{n\in\Z}f[n]z^{-n-1}$.
It is graded by weight $m$ and degree $d$. Namely, we have the decomposition
$V^k_l=\bigoplus_{m,d=0}^\infty\bigl(V^k_l\bigr)_{m,d}$,
where $\left(V^k_l\right)_{m,d}$ is spanned by the monomial vectors
\bea\label{MON}
f[0]^{a_0}f[-1]^{a_1}f[-2]^{a_2}\cdots v^k_l
\ena
with
\bea\label{DEG}
\sum_{j=0}^\infty a_j=m,\quad\sum_{j=0}^\infty ja_j=d.
\ena
A basis of the principal subspace $V^k_l$ is described by the set
of integer points $\P^k_l={^{\mathbb R}\P^k_l}\cap\Z^\infty$
where
\be
\qquad{^{\mathbb R}\P^k_l}=\{(a_n)_{n\geq0}\mid
a_n\in{\mathbb R}_{\geq0},\ a_n=0
\hbox{ for almost all $n$},\
a_j+a_{j+1}\leq k\ (j\geq0),\ a_0\leq l\}.
\en
\begin{prop}\cite{FS}
The set of monomial vectors \eqref{MON} with the exponents $(a_n)_{n\geq0}$
taken from the set $\P^k_l$ constitutes a basis of $V^k_l$.
\end{prop}
The condition $a_j+a_{j+1}\leq k$ comes from the integrability condition
$f(z)^{k+1}=0$.

The character of $V^k_l$ is the formal power series
\be
\chi^k_l(z)=\sum_{m,d=0}^\infty\dim\bigl(V^k_l\bigr)_{m,d}z^mq^d.
\en
Two different formulas are known for this quantity.
\begin{prop}\cite{FS}
The character $\chi^k_l(z)$ is given by
\bea\label{FF}
\chi^k_l(z)=\sum_{m_1,\ldots,m_k=0}^\infty
\frac{q^{\sum_{i,j=1}^km_im_j\min(i,j)-\sum_{i=1}^k\min(i,l)m_i}
z^{\sum_{i=1}^kim_i}}{(q)_{m_1}\cdots(q)_{m_k}}.
\ena
\end{prop}
This formula is called fermionic.
\begin{prop}\cite{FL}
The character $\chi^k_l(z)$ is given by
\bea
\chi^k_l(z)=\sum_{n=0}^\infty\frac{z^{nk}q^{n^2k-nl}}
{(q^{2n}z)_\infty(q)_n(q^{-2n+1}z^{-1})_n}
+\sum_{n=0}^\infty\frac{z^{nk+l}q^{n^2k+nl}}
{(q^{2n+1}z)_\infty(q)_n(q^{-2n}z^{-1})_{n+1}}.\label{BF}
\ena
\end{prop}
This formula is called bosonic.
\subsection{Recursion relation}
The recursion relation for the characters
\be
\chi^k_l(z)=\sum_{a=0}^lz^a\chi^k_{k-a}(qz)
\en
follows immediately from the definition of $\P^k_l$.
It can be rewritten as
\bea\label{REC}
\chi^k_l(z)=\chi^k_{l-1}(z)+z^l\chi^k_{k-l}(qz).
\ena
In this form, the recursion can be explained by
the representation theory as follows.
Let $V^k_l[i]$ be the $U(\mathfrak n)$-module identical with $V^k_l$
as a vector space, where $f[n]$ acts as
$f[n+i]$. We denote the vector $v^k_l$
considered as the cyclic vector of the module $V^k_l[i]$ by $v^k_l[i]$.
The identity \eqref{REC} for the characters
corresponds to a short exact sequence of $U(\hat{\mathfrak n})$-modules.
\begin{prop}
There is an exact sequence of $U(\hat{\mathfrak n})$ modules
\be
0\rightarrow V^k_{k-l}[-1]
\buildrel\iota\over\rightarrow V^k_l[0]
\rightarrow V^k_{l-1}[0]
\rightarrow0,
\en
where the homomorphism $\iota$ is defined by
\be
\iota(v^k_{k-l}[-1])=f[0]^lv^k_l[0].
\en
\end{prop}
\begin{proof}
Let $\br{f[0]^l}$ be the submodule of $V^k_l[0]$ generated by
the vector $f[0]^lv^k_l[0]$. By the definition, there is an isomorphism
\be
V^k_l[0]/\br{f[0]^l}\simeq V^k_{l-1}[0].
\en
On the other hand, the integrability condition $f(z)^{k+1}=0$
implies
\be
f[-1]^{k-l}f[0]^lv^k_l[0]=0.
\en
Therefore, the mapping $\iota$ is well-defined.
The character identity implies
that it is an inclusion.
\end{proof}
\subsection{Contributions of extremal vectors}
Let $\P$ be a convex subset of a real vector space.
A point in $\P$ is called extremal if and only if it is not
a linear combination $\theta P_1+(1-\theta)P_2$ $(0<\theta<1)$
of two distinct points $P_1,P_2\in\P$.
A general principle in ``counting'' the number of integer points
in convex polygons
is to write it as the sum of contributions from the extremal points.
In this subsection, we show how we can guess the bosonic formula
\eqref{BF} by using this principle.

Weyl's character
formula is of this kind: the simplest case is
\bea\label{EXA}
\frac{1-z^{l+1}}{1-z}=\frac1{1-z}+\frac{z^l}{1-z^{-1}}.
\ena
The polygon in this example is the interval $[0,l]$.
{}For the character, instead of counting the number of integer points,
we count $z^n$ for each integer point $n$ in $[0,l]$.
The extremal points are
$0$ and $l$. The contribution from an extremal point
is counted as the sum of $z^n$ over the integer points near that point
in the limit $l\rightarrow\infty$. For $0$ this is $1+z+z^2+\cdots=1/(1-z)$,
and for $l$ this is $z^l+z^{l-1}+\cdots=z^l/(1-z^{-1})$. These are the two
terms in the right hand side of \eqref{EXA}. To obtain
the left hand side of the formula,
we write the second term as $-z^{l+1}/(1-z)$.
Because of rewriting $z^l+z^{l-1}+\cdots$ to $-z^{l+1}-z^{l+2}-\cdots$,
the obtained formula contains both positive and negative coefficient terms.
The formula \eqref{BF} should be understood as an equality of
non-negative power series
in $z$. The $n$-th summand of the first sum (respectively, of the second sum)
in the right hand side contains negative coefficient terms
if and only if $n$ is odd (respectively, even).
This is the difference between the bosonic formula \eqref{BF}
and the fermionic one \eqref{FF}.
In the latter, each term corresponds to an integer point, and therefore,
the formula consists of positive coefficient terms only.

There is an important point in counting contributions of extremal points.
{}For the characters of the principal subspaces, we count monomials
$z^mq^d$ using two linear functions \eqref{DEG}. It means that we
count not the integer points in the infinite dimensional polygon
${^{\mathbb R}\P^k_l}$ but rather the integer points with multiplicities in
a polygon in $\mathbb R^2$, which is the image of ${^{\mathbb R}\P^k_l}$
by the mapping $(m,d)$. Not all extremal points remain to be extremal
when they are projected to $\mathbb R^2$. We guess that the contributions
from such extremal points in ${^{\mathbb R}\P^k_l}$ that are projected
to a non-extremal point in the image cancel out. Thus we count only
the contributions from the extremal points in the image.

In the case of the Weyl-Kac character formula for integrable modules,
the relevant extremal points in the weight space
are given by the Weyl group orbit of the highest weight.
In the case of the principal subspaces,
we also take the extremal points in 
${^{\mathbb R}\P^k_l}$ whose weights belong to
the orbit of the highest weight. They are
\bea
&&w_{2n}:a_0=l,a_1=k-l,\ldots,a_{2n-2}=l,a_{2n-1}=k-l,a_j=0\ (j\geq2n),
\label{extw}\\
&&w_{2n+1}:a_0=l,a_1=k-l,\ldots,a_{2n-1}=k-l,a_{2n}=l,
a_j=0\ (j\geq2n+1),
\label{extw2}
\ena
where $n\geq0$. The monomials $z^{m(w)}q^{d(w)}$ at these points are given by
\be
z^{m(w)}q^{d(w)}=
\begin{cases}
z^{nk}q^{n^2k-nl}&\hbox{ for }w=w_{2n};\\
z^{nk+l}q^{n^2k+nl}&\hbox{ for }w=w_{2n+1}.
\end{cases}
\en

The contribution from an extremal point is defined to be
the formal series obtained in the limit
\be
z^{m(w)}q^{d(w)}\times\lim_{l,k-l\rightarrow\infty}
z^{-m(w)}q^{-d(w)}\chi^k_l(z)
\en
The contribution can be calculated in several different ways.
Here we present two such calculations.

A direct calculation using $\P^k_l$ is possible for the present case
of $\widehat{\mathfrak{sl}}_2$.
In the limit the shape of ${^{\mathbb R}\P^k_l}$ in the vicinity of each
extremal point $w$ becomes a cone of the form $w+{^{\mathbb R}\mathcal C}$.
{}For the extremal point
\be
w=w_{2n}=(l,k-l,l,k-l,\ldots,l,k-l,0,0,0,\ldots),
\en
the cone ${^{\mathbb R}\mathcal C}$ is generated by the following vectors:
\newcommand{\zero}{\phantom{-}0}
\newcommand{\one}{\phantom{-}1}
\be
&(\zero,\zero,\zero,\zero,\ldots,-1,\one,\zero,\zero,\zero,\ldots),&q,\\
&\qquad\cdots&\\
&(\zero,\zero,-1,\one,\ldots,-1,\one,\zero,\zero,\zero,\ldots),&q^{n-1},\\
&(-1,\one,-1,\one,\ldots,-1,\one,\zero,\zero,\zero,\ldots),&q^n,\\
&
(\zero,\zero,\zero,\zero,\ldots,\zero,-1,\zero,\zero,\zero,\ldots),&q^{-2n+1}z^{-1},\\
&\qquad\cdots&\\
&
(\zero,\zero,\zero,-1,\ldots,\one,-1,\zero,\zero,\zero,\ldots),&q^{-n-1}z^{-1},\\
&
(\zero,-1,\one,-1,\ldots,\one,-1,\zero,\zero,\zero,\ldots),&q^{-n}z^{-1},\\
&(\zero,\zero,\zero,\zero,\ldots,\zero,\zero,\one,\zero,\zero,\ldots),&q^{2n}z,\\
&
(\zero,\zero,\zero,\zero,\ldots,\zero,\zero,\zero,\one,\zero,\ldots),&q^{2n+1}z,\\
&\qquad\cdots,&
\en
where in the first line $-1$ stands at the $(2n-1)$-th position.
After each vector we wrote the corresponding monomial.
Note that all these vectors are linearly independent and
the integer points in the cone, $\mathcal C=
{^{\mathbb R}\mathcal C}\cap \Z^\infty$, are linear combinations of the
generating vectors with non-negative integer coefficients.
Therefore, the contribution of the extremal point $w_{2n}$ is the sum of
products of powers of these monomials.
Thus we get the summands in the first term
in the right hand side of \eqref{BF}. The case of $w=w_{2n+1}$ is similar.

The second calculation of the contribution
uses the fermionic formula. In the fermionic formula, terms
are labeled by $(m_1,\ldots,m_k)$. In this case, the term corresponding to
each point is not a monomial but a formal power series.
It is written explicitly in \eqref{FF}. It is known that
there is a mapping $\Phi^k:\P^k_k\rightarrow\Z_{\geq0}^k$ such that
the sum of monomials over $(\Phi^k)^{-1}(m_1,m_2,\ldots,m_k)\cap\P^k_l$
is the series corresponding to $(m_1,m_2,\ldots,m_k)$.
The mapping is defined inductively as follows.
Let $a=(a_n)_{n\geq0}\in\P^k_k$.
If $a_i+a_{i+1}<k$ for all $i$ we have $a\in\P^{k-1}_{k-1}$,
and we define $\Phi^k(a)=(\Phi^{k-1}(a),0)$.
If $a_i+a_{i+1}=k$ for some $i$, define $b\in\P^k_k$ by
\be
b_j=\begin{cases}a_j&\quad(j<i);\\
a_{j+2}&\quad(j\geq i).\end{cases}
\en
Then, we define $\Phi^k(a)=\Phi^k(b)+(0,0,\cdots,0,1)$.
{}For example, consider the extremal points 
$w_{2n+\epsilon}\in {^{\mathbb R}\P^k_l}$
given by \eqref{extw}, \eqref{extw2}. We have 
\be
\Phi^k(w_{2n+\epsilon})
=(0,0,\ldots,0,\buildrel{l-{\rm th}}\over\epsilon,0,\ldots,
\buildrel{k-{\rm th}}\over n).
\en
The points of the cone ${^{\mathbb R}\mathcal{C}}$ 
are linear combinations of the generating vectors. 
{}For $M\in \Z>0$, let $w_{2n+\epsilon}+\mathcal C(M)$ be the 
subset of $w_{2n+\epsilon}+\mathcal C$ such that 
the sum of the coefficients in front of the generating vectors 
is bounded by $M$.
Set
\be
&&\mathcal M^n_\epsilon(N)=\{(m_1,m_2,m_3,\cdots,m_{l-2},m_{l-1},
m_l,m_{l+1},m_{l+2},\cdots,m_{k-2},m_{k-1},m_k)\\
&&\qquad|\sum_{0\leq i\leq N}m_{k-i}=n,\sum_{|i|\leq N}m_{l+i}=\epsilon\}.
\en
{}For any $M$, if $k-l-N,l-N,N$ are large enough,
the points in $w_{2n+\epsilon}+\mathcal C(M)$ are mapped to 
$\mathcal M^n_\epsilon(N)$. Thus, the contribution
of the extremal point $w_{2n+\epsilon}$ is equal to the sum of series
corresponding to the points in $\mathcal M^n_\epsilon(N)$ in
the limit $k-l-N,l-N,N\rightarrow\infty$.

Let us check this statement by direct calculation.
In the limit, the sum splits into the product of
sums over the following three sets of indices.

\smallskip
(i)\ ${\mathcal M}_0=\{(m_1,m_2,\ldots)\mid m_j\geq0\}$;

\smallskip
(ii)\
${\mathcal M}_1=\Bigl\{(\ldots,m_{l-1},m_l,m_{l+1},\ldots)\,\Big|\,\sum_{i\in\Z}m_{l+i}
=\epsilon\Bigr\}$,
\quad($\epsilon=0$ for $w_{2n}$ and $\epsilon=1$ for $w_{2n+1}$);

\smallskip
(iii)\
${\mathcal M}_2=\left\{(\ldots,m_{k-1},m_k)\,\Big|\,\sum_{i\geq0}m_{k-i}=n\right\}$.

\smallskip\noindent
Consider the exponent of $q$ in \eqref{FF}. By using the equalities
$\sum_{i\in\Z}m_{l+i}=\epsilon$ or $\sum_{i\geq0}m_{k-i}=n$,
one can rewrite it so that the sums over the three sets become independent.

The sum over (i) is
\be
&&\sum_{m\in {\mathcal M}_0}q^{\sum_{i,j=1}^\infty m_im_j\min(i,j)
+\{2(\sum_{i\in\Z}m_{l+i}+\sum_{i\geq0}m_{k-i})-1\}\sum_{j=1}^\infty jm_j}/
\prod_{i\geq1}(q)_{m_i}\\
&&=\sum_{m\in {\mathcal M}_0}q^{\sum_{i,j=1}^\infty m_im_j\min(i,j)
+\{2(n+\epsilon)-1\}\sum_{j=1}^\infty jm_j}/
\prod_{i\geq1}(q)_{m_i}\\
&&=\frac1{(q^{2(n+\epsilon)}z)_\infty}.
\en

The sum over (ii) with $\epsilon=1$ is split into two parts:
\be
{\mathcal M}_1=\{(\ldots,0,\buildrel\hbox{$\scriptstyle(l-i)$th}\over1,0\ldots,0)\mid i\geq0\}
\sqcup
\{(0,\ldots,0,\buildrel\hbox{$\scriptstyle(l+i)$th}\over1,0,\ldots)\mid i\geq1\}.
\en
It is evaluated as follows.
\be
&&\frac1{(q)_1}\sum_{i=0}^\infty q^{(l-i)(2n+1)-(l-i)}z^{l-i}
+\frac1{(q)_1}\sum_{i=1}^\infty q^{(l+i)(2n+1)-l}z^{l+i}\\
&&\qquad=\frac{q^{2nl}z^l}{(q)_1(q^{-2n}z^{-1})_1}
+\frac{q^{2nl+2n+1}z^{l+1}}{(q)_1(q^{2n+1}z)_1}\\
&&\qquad=\frac{q^{2nl}z^l}{(q^{-2n}z^{-1})_1(q^{2n+1}z)_1}.
\en
{}For $\epsilon=0$ the sum is over one point and is equal to $1$.

The sum over (iii) is more complicated. We state the result.
\begin{lem}
Set
\be
g_{n,k}(z)=\sum_{(n_0,n_1,\ldots),\ \sum_{j=0}^\infty n_j=n}
\frac{q^{\sum_{i,j=0}^\infty n_in_j\min(k-i,k-j)}z^{\sum_{i=0}^\infty(k-i)n_i}}
{\prod_{i=0}^\infty(q)_{n_i}}.
\en
Then we have the equality
\bea\label{FNK}
g_{n,k}(z)=\frac{q^{n^2k}z^{nk}}{(q)_n(q^{-2n+1}z^{-1})_n}.
\ena
\end{lem}
\begin{proof}
Splitting the sum in $g_{n,k}(z)$ into $n+1$ subsums corresponding to
$n_0=0,1,\ldots,n$,
we obtain
\be
g_{n,k}(z)=\sum_{i=0}^n\frac{q^{i^2k}z^{ik}}{(q)_i}g_{n-i,k-1}(q^{2i}z).
\en
This recursion has a unique solution of the form $g_{n,k}(z)\in 1+z\C[[z]]$.
One can check the recursion is satisfied by \eqref{FNK}.
The check reduces to the equality
\be
\sum_{i=0}^n(-1)^i\frac{(q)_n}{(q)_i(q)_{n-i}}(q^nz)_iq^{i(i+1)/2-in}
=q^{n^2}z^n.
\en
\end{proof}
Combining (i), (ii), (iii), we obtain the summands in \eqref{BF}.


\section{Highest weight $\nh$-modules}
In this section, we introduce a family of modules which
generalize the principal subspaces of integrable
$\widehat{\mathfrak{sl}}_3$-modules.

We fix the  notation as follows.
Let  $\nh=\n\otimes \C[t,t^{-1}]$ denote
the current algebra over the nilpotent subalgebra
$\n=\C e_{21}\oplus \C e_{31}\oplus \C e_{32}$
of $\mathfrak{sl}_3$.
The basis elements
$e_{ab}[n]=e_{ab}\otimes t^n$ ($1\le b<a\le 3$, $n\in\Z$)
of $\nh$ satisfy the relations
\be
[e_{32}[m],e_{21}[n]]=e_{31}[m+n],\quad
[e_{21}[m],e_{31}[n]]=0,\quad
[e_{32}[m],e_{31}[n]]=0.
\en
We set
\be
e_{ab}(z)=\sum_{n\in\Z}e_{ab}[n]z^{-n-1}.
\en
With the degree assignment
\be
\deg e_{21}[n]=(1,0,-n), \quad
\deg e_{32}[n]=(0,1,-n),
\en
$\nh$ is a $\Z_{\ge0}^2\times\Z$-graded Lie algebra.
All $\nh$-modules considered in this paper are
graded $\nh$-modules. 

\subsection{Principal subspaces for $\widehat{\mathfrak{sl}_3}$}

{}For non-negative integers $k,l_1,l_2$ satisfying $l_1+l_2\le k$,
let $M^{k}_{l_1,l_2}$ be the level $k$
integrable highest weight $\widehat{\mathfrak{sl}}_3$-module
of highest weight $(l_1,l_2)$.
The highest weight vector $v^k_{l_1,l_2}$ is characterized 
up to scalar multiple by 
\be
&&x[n]v^k_{l_1,l_2}=0\quad (x\in {\mathfrak{sl}}_3;n>0),
\\
&&e_{ab}[0]v^k_{l_1,l_2}=0\quad (a<b),
\\
&&(e_{11}[0]-e_{22}[0])v^k_{l_1,l_2}=l_1 v^k_{l_1,l_2},
\quad 
(e_{22}[0]-e_{33}[0])v^k_{l_1,l_2}=l_2 v^k_{l_1,l_2},
\\
&&e_{21}[0]^{l_1+1}v^k_{l_1,l_2}=0,
\quad
e_{32}[0]^{l_2+1}v^k_{l_1,l_2}=0.
\en
We call the $U(\nh)$-submodule 
\be
V^k_{l_1,l_2}=U(\nh)\cdot v^k_{l_1,l_2}\subset M^{k}_{l_1,l_2}
\en
the {\it principal subspace} of $M^{k}_{l_1,l_2}$.
The following relations for $v=v^k_{l_1,l_2}$ 
take place in $V^k_{l_1,l_2}$:
\bea
&&e_{21}[n]v=0\quad (n>0), \label{BW1}\\
&&e_{31}[n]v=0\quad (n>0), \label{BW2}\\
&&e_{32}[n]v=0\quad (n>0), \label{BW3}\\
&&e_{21}[0]^{l_1+1}v=0, \label{BW4}\\
&&e_{32}[0]^{l_2+1}v=0, \label{BW5}\\
&&e_{21}(z)^{k+1}=0,\quad e_{32}(z)^{k+1}=0\,.
\label{BW6}
\ena
We remark that also
\bea
e_{21}[0]^\alpha e_{31}[0]^{l_1+l_2-\alpha+1}v=0
\qquad
(0\le\alpha\le l_1+l_2+1)
\label{BW7}
\ena
holds, due to the following lemma.

\begin{lem}\label{lem:xyz}
Let $w$ be a vector in an $\mathfrak{n}$-module such that
$e_{21}^{l_1+1}w=0$ and $e_{32}^{l_2+1}w=0$
for some non-negative integers $l_1,l_2$.
Then $e_{21}^\alpha e_{31}^{\beta}w=0$ holds for all
$\alpha,\beta\ge0$ with $\alpha+\beta=l_1+l_2+1$.
\end{lem}
\begin{proof}
Let $W$ be the irreducible $\mathfrak{sl}_3$-module
with highest weight $(l_1,l_2)$.
It is simple to check that
the lemma holds for the highest weight vector of $W$.
On the other hand,
$W$ is isomorphic to
the quotient of the free left $U(\mathfrak{n})$-module
$U(\mathfrak{n})$ by the submodule generated by
$e_{21}[0]^{l_1+1}$, $e_{32}[0]^{l_2+1}$.
Hence $U(\mathfrak{n})\cdot w$ is a quotient of $W$, and the assertion follows.
\end{proof}
Our goal is to find a
formula for the character of $V^k_{l_1,l_2}$.

\subsection{Modules $U^{k_1,k_2}_{l_1,l_2,l_3}$, $V^{k_1,k_2}_{l_1,l_2,l_3}$}

We shall introduce a family of cyclic $\nh$-modules
$U^{k_1,k_2}_{l_1,l_2,l_3}$, $V^{k_1,k_2}_{l_1,l_2,l_3}$
parametrized by non-negative integers $k_1,k_2$ and $l_1,l_2,l_3$.
The parameters $k_1,k_2$ play a role similar to that
of the level of representations,
while $l_1,l_2,l_3$ correspond to the highest weight.
Throughout the paper we assume that
\be
k_1\le k_2.
\en

\begin{dfn}
We define
$V^{k_1,k_2}_{l_1,l_2,l_3}$
to be the $\nh$-module
generated by a non-zero vector $v$ under the following defining relations:
\bea
&&e_{21}[n]v=0\quad (n>0),\label{BV1}\\
&&e_{31}[n]v=0\quad (n>0),\label{BV2}\\
&&e_{32}[n]v=0\quad (n>0),\label{BV3}\\
&&e_{21}[0]^{l_1+1}v=0,\label{V1}\\
&&e_{21}[0]^\alpha e_{31}[0]^{l_3+1-\alpha}v=0
\quad (0\le\alpha\le l_3+1),
\label{V2}\\
&&e_{32}[0]^{l_2+1}v=0,\label{V3}
\\
&&e_{21}(z)^{k_1+1}=0,\quad
e_{32}(z)^{k_2+1}=0\,.
\label{V4}
\ena
\end{dfn}

\begin{dfn}
We define
$U^{k_1,k_2}_{l_1,l_2,l_3}$
to be the $\nh$-module generated by a non-zero vector
$v$ with the following defining relations:
\bea
&&e_{21}[n]v=0\quad (n>0),\label{N1}\\
&&e_{31}[n]v=0\quad (n>1),\label{N2}\\
&&e_{32}[n]v=0\quad (n>0),\label{N3}\\
&&e_{31}[1]^{l_3+1}v=0,\label{U1}\\
&&e_{21}[0]^\alpha e_{31}[1]^{l_1+1-\alpha}v=0
\quad (0\le \alpha\le l_1+1),\label{U2}\\
&&e_{32}[0]^{l_2+1}v=0,\label{U3}\\
&&e_{21}(z)^{k_1+1}=0,\quad
e_{32}(z)^{k_2+1}=0\,.
\label{U4}
\ena
\end{dfn}
Taking commutators of $e_{32}[0]$ and \eqref{V4} or \eqref{U4},
we obtain also the relations
\bea
e_{21}(z)^\alpha e_{31}(z)^{k_1-\alpha+1}=0
\qquad (0\le \alpha\le k_1+1).
\label{k3}
\ena
We use the following notation for the characters of these modules:
\be
&&
\psi^{k_1,k_2}_{l_1,l_2,l_3}(z_1,z_2)
={\rm ch}_{z_1,z_2,q} V^{k_1,k_2}_{l_1,l_2,l_3},
\\
&&
\varphi^{k_1,k_2}_{l_1,l_2,l_3}(z_1,z_2)
={\rm ch}_{z_1,z_2,q} U^{k_1,k_2}_{l_1,l_2,l_3}.
\en
Note that our characters are normalized in such a way that the
degree of the cyclic vectors is $(0,0,0)$.

\begin{rem}\label{remV=V}
{}From the definition it readily follows that
$V^{k,k}_{l_1,l_2,l_1+l_2}\simeq
U^{k,k}_{l_1,l_2,0}$.
The principal subspace $ V^k_{l_1,l_2}$ is its quotient.
Indeed, comparing \eqref{BV1}--\eqref{V3} with
\eqref{BW1}--\eqref{BW6} and the remark after that,
we see that there is a surjection of $\nh$-modules
\be
V^{k,k}_{l_1,l_2,l_1+l_2}\to V^k_{l_1,l_2}\to 0.
\en
Later it will turn out to be an isomorphism
(see Corollary \ref{V=Vpr}).
\end{rem}

Some of the modules are the same which
follows immediately from the definition.

\begin{lem}\label{lem:V=V}
We have
\be
V^{k_1,k_2}_{l_1,l_2,l_3}
&=&V^{k_1,k_2}_{k_1,l_2,l_3} \quad (l_1>k_1), \\
V^{k_1,k_2}_{l_1,l_2,l_3}&=&V^{k_1,k_2}_{l_1,k_2,l_3} \quad (l_2>k_2), \\
V^{k_1,k_2}_{l_1,l_2,l_3}&=&V^{k_1,k_2}_{l_1,l_2,k_1} \quad (l_3>k_1), \\
V^{k_1,k_2}_{l_1,l_2,l_3}&=&V^{k_1,k_2}_{l_3,l_2,l_3} \quad (l_1>l_3), \\
V^{k_1,k_2}_{l_1,l_2,l_3}&=&V^{k_1,k_2}_{l_1,l_2,l_1+l_2}\quad (l_3>l_1+l_2)\,.
\en
\end{lem}
\begin{proof}
{}For example,
\eqref{V1}, \eqref{V3} and Lemma \ref{lem:xyz} imply that
$e_{21}[0]^\alpha e_{31}[0]^{l_1+l_2-\alpha+1}v=0$
for all $0\le\alpha\le l_1+l_2+1$.
This proves the last relation. Other cases can be verified
similarly using either the definition,
\eqref{V4} or \eqref{k3}.
\end{proof}

\begin{lem}\label{lem:U=U}
We have
\be
U^{k_1,k_2}_{l_1,l_2,l_3}
&=&U^{k_1,k_2}_{k_1,l_2,l_3} \quad (l_1>k_1), \\
U^{k_1,k_2}_{l_1,l_2,l_3}&=&U^{k_1,k_2}_{l_1,k_2,l_3} \quad (l_2>k_2), \\
U^{k_1,k_2}_{l_1,l_2,l_3}&=&U^{k_1,k_2}_{l_1,l_2,\min(l_1,l_2)} \quad (l_3>\min(l_1,l_2))\,.
\en
\end{lem}
\begin{proof}
By \eqref{U2} we have $e_{31}[1]^{l_1+1}v=0$.
Taking
the commutator of $e_{21}[1]$ with $e_{32}[0]^{l_2+1}$
and using $e_{21}[1]v=0$ and \eqref{U3},
we obtain  $e_{31}[1]^{l_2+1}v=0$. This proves the third relation.
The other relations are proved similarly.
\end{proof}

In view of Lemmas \ref{lem:V=V}--\ref{lem:U=U},
we may restrict our attention to the modules
\be
\{U^{k_1,k_2}_{l_1,l_2,l_3}\mid
(l_1,l_2,l_3)\in P_U^{k_1,k_2} \},
\quad
\{V^{k_1,k_2}_{l_1,l_2,l_3}\mid
(l_1,l_2,l_3)\in P_V^{k_1,k_2}\},
\en
where the parameter regions are the following subsets of
$\Z^3_{\ge 0}$:
\bea
&&P_U^{k_1,k_2}=\{(l_1,l_2,l_3)\mid
0\leq l_1\leq k_1,\ 0\leq l_2\leq k_2,\ 0\leq l_3\leq
\min(l_1,l_2)\},
\label{Pphi}
\\
&&P_V^{k_1,k_2}=\{(l_1,l_2,l_3)\mid
0\leq l_1\leq k_1,\ 0\leq l_2\leq k_2,
\ l_1\leq l_3\leq\min(l_1+l_2,k_1)\}.
\label{Ppsi}
\ena

\subsection{Subquotient modules}

In this subsection we study a recurrent
structure for some subquotients of
$U^{k_1,k_2}_{l_1,l_2,l_3}$, $V^{k_1,k_2}_{l_1,l_2,l_3}$.
Denote by $T_{m,n}$ the automorphism of $\nh$ given by
\be
T_{m,n} e_{21}[i]=e_{21}[i-m],
\quad
T_{m,n} e_{31}[i]=e_{31}[i-m-n],
\quad
T_{m,n} e_{32}[i]=e_{32}[i-n].
\en
{}For an $\nh$-module $M$, we denote by $M[m,n]$ the module
with the same underlying vector space
on which $x\in\nh$ acts as $T_{m,n}x$.
{}For a cyclic $\nh$-module $M$ with a cyclic vector $v$ and
$f\in U(\nh)$, we use the notation
$$
\langle f\rangle=U(\nh)\cdot fv.
$$
In what follows, we set
$U^{k_1,k_2}_{l_1,l_2,l_3}=0$,
$V^{k_1,k_2}_{l_1,l_2,l_3}=0$
if one of $l_i$'s is negative.

\begin{lem}\label{lem:exact1}
We have an exact sequence of $\nh$-modules
\be
&&V^{k_1,k_2}_{l_3,k_2-l_2,l_1}[1,-1]\rightarrow
U^{k_1,k_2}_{l_1,l_2,l_3}\rightarrow
U^{k_1,k_2}_{l_1,l_2-1,l_3}\rightarrow0.
\en
\end{lem}
\begin{proof}
Consider the submodule $\langle e_{32}[0]^{l_2}\rangle$ of
$U^{k_1,k_2}_{l_1,l_2,l_3}$.
By the definition we have an exact sequence
\be
0\rightarrow
\langle e_{32}[0]^{l_2}\rangle
\rightarrow
U^{k_1,k_2}_{l_1,l_2,l_3}\rightarrow
U^{k_1,k_2}_{l_1,l_2-1,l_3}\rightarrow0.
\label{short1}
\en
We show that there is a surjection
\be
V^{k_1,k_2}_{l_3,k_2-l_2,l_1}[1,-1]\rightarrow
\langle e_{32}[0]^{l_2}\rangle 
\rightarrow0.
\label{short2}
\en
It suffices to verify the following relations for
$v_1=e_{32}[0]^{l_2}v$ where $v$ denotes the generator of
$U^{k_1,k_2}_{l_1,l_2,l_3}$:
\bea
&&e_{21}[n]v_1=0\quad (n>1),\label{P1}\\
&&e_{31}[n]v_1=0\quad (n>0),\label{P2}\\
&&e_{32}[n]v_1=0\quad (n>-1),\label{P3}\\
&&e_{21}[1]^{l_3+1}v_1=0,\label{P4}\\
&&e_{21}[1]^\alpha e_{31}[0]^{l_1+1-\alpha}v_1=0,\label{P5}\\
&&e_{32}[-1]^{k_2-l_2+1}v_1=0.\label{P6}
\ena
Equation \eqref{P1} follows from
$e_{21}[n]v=0$, $[e_{32}[0],e_{21}[n]]=e_{31}[n]$,
$e_{31}[n]v=0$ and $[e_{31}[n],e_{32}[0]]=0$
with $n>1$.
Equation \eqref{P2} follows from \eqref{N2} for $n\geq2$.
{}For $n=1$ it follows from $e_{21}[1]v=0$ and \eqref{U3}
by using $[e_{32}[0],e_{21}[1]]=e_{31}[1]$.
Equation \eqref{P3} follows from \eqref{N3} and \eqref{U3}.
Equation \eqref{P4} follows from \eqref{U1}, $e_{21}[1]v=0$
and $[e_{32}[0],e_{21}[1]]=e_{31}[1]$.

Let us prove \eqref{P5}. From \eqref{U2} we have
\be
e_{21}[0]^{l_1+1-\alpha}e_{31}[1]^\alpha v=0.
\en
{}From \eqref{U3}, using $e_{21}[1]v=0$, we obtain
\be
e_{32}[0]^{l_2+1-\alpha}e_{31}[1]^\alpha v=0.
\en
A variant of Lemma \ref{lem:xyz}
and the last two equations above
lead to
\be
e_{31}[0]^{l_1-\alpha}e_{32}[0]^{l_2-\alpha}e_{31}[1]^\alpha v=0.
\en
It follows that
\be
e_{21}[1]^\alpha e_{31}[0]^{l_1+1-\alpha}e_{32}[0]^{l_2}v=
e_{31}[0]^{l_1-\alpha}e_{31}[1]^\alpha e_{32}[0]^{l_2-\alpha}v=0.
\en
{}Finally, \eqref{P6} is a consequence of $e_{32}(z)^{k_2+1}=0$.
\end{proof}

\begin{lem}\label{lem:exact2}
We have an exact sequence
\be
&&
U^{k_1,k_2}_{l_3,k_2-l_2,l_1}[0,-1]\rightarrow
V^{k_1,k_2}_{l_1,l_2,l_3}\rightarrow
V^{k_1,k_2}_{l_1,l_2-1,l_3}\rightarrow0.
\en
\end{lem}
\begin{proof}
We consider the submodule $\langle e_{32}[0]^{l_2}\rangle$
of $V^{k_1,k_2}_{l_1,l_2,l_3}$, so that
\be
0\rightarrow
\langle e_{32}[0]^{l_2}\rangle
\rightarrow
V^{k_1,k_2}_{l_1,l_2,l_3}\rightarrow
V^{k_1,k_2}_{l_1,l_2-1,l_3}\rightarrow0
\en
is exact. We then show that there is a surjection
\be
&&
U^{k_1,k_2}_{l_3,k_2-l_2,l_1}[0,-1]
\to
\langle e_{32}[0]^{l_2}\rangle
\to 0
\en
by checking the defining relations.
The rest of the
proof is similar to that of Lemma \ref{lem:exact1}, we omit
the details.
%

\end{proof}

\begin{lem}\label{lem:exact3}
We have exact sequences
\bea
&&
V^{k_1,k_2}_{l_1-l_3,l_2-l_3,k_1-l_3}\rightarrow
U^{k_1,k_2}_{l_1,l_2,l_3}\rightarrow
U^{k_1,k_2}_{l_1,l_2,l_3-1}\rightarrow0\,,
\label{exact3}
\\
&&
U^{k_1,k_2}_{k_1-l_1,l_1+l_2,l_3-l_1}[-1,0]\rightarrow
V^{k_1,k_2}_{l_1,l_2,l_3}\rightarrow
V^{k_1,k_2}_{l_1-1,l_2,l_3}\rightarrow0\,.
\label{exact4}
\ena
\end{lem}
\begin{proof}
We repeat the argument of
Lemmas \ref{lem:exact1} and \ref{lem:exact2}.

To show \eqref{exact3},
we take the submodule generated by
$e_{31}[1]^{l_3}v\in U^{k_1,k_2}_{l_1,l_2,l_3}$
and check the exact sequence
\be
&&0\to
\langle e_{31}[1]^{l_3}\rangle
\to
U^{k_1,k_2}_{l_1,l_2,l_3}\rightarrow
U^{k_1,k_2}_{l_1,l_2,l_3-1}\rightarrow0,
\\
&&
V^{k_1,k_2}_{l_1-l_3,l_2-l_3,l_1+l_2-2l_3}\rightarrow
\langle e_{31}[1]^{l_3}\rangle
\to 0.
\en

{}For the proof of \eqref{exact4},
we take $e_{21}[0]^{l_1}v\in V^{k_1,k_2}_{l_1,l_2,l_3}$
and verify
\be
&&0\to
\langle e_{21}[0]^{l_1} \rangle
\to
V^{k_1,k_2}_{l_1,l_2,l_3}\rightarrow
V^{k_1,k_2}_{l_1-1,l_2,l_3}\rightarrow0,
\\
&&
U^{k_1,k_2}_{k_1-l_1,l_1+l_2,l_3-l_1}[-1,0]
\rightarrow
\langle e_{21}[0]^{l_1} \rangle
\to 0.
\en
\end{proof}

{}From Lemmas \ref{lem:exact1}--\ref{lem:exact3},
we obtain
the following upper estimate for the characters
of all modules in the parameter regions
\eqref{Pphi},\eqref{Ppsi}  :
\bea
&&\varphi^{k_1,k_2}_{l_1,l_2,l_3}(z_1,z_2)
\le
\varphi^{k_1,k_2}_{l_1,l_2-1,l_3}(z_1,z_2)+
z_2^{l_2}\psi^{k_1,k_2}_{l_3,k_2-l_2,l_1}(q^{-1}z_1,qz_2),
\label{Ra}\\
&&\psi^{k_1,k_2}_{l_1,l_2,l_3}(z_1,z_2)
\le\psi^{k_1,k_2}_{l_1,l_2-1,l_3}(z_1,z_2)
+
z_2^{l_2}\varphi^{k_1,k_2}_{l_3,k_2-l_2,l_1}(z_1,qz_2),
\label{Rb}\\
&&\varphi^{k_1,k_2}_{l_1,l_2,l_3}(z_1,z_2)
\le\varphi^{k_1,k_2}_{l_1,l_2,l_3-1}(z_1,z_2)
+(q^{-1}z_1z_2)^{l_3}
\psi^{k_1,k_2}_{l_1-l_3,l_2-l_3,k_1-l_3}(z_1,z_2),
\label{Rc}\\
&&\psi^{k_1,k_2}_{l_1,l_2,l_3}(z_1,z_2)
\le\psi^{k_1,k_2}_{l_1-1,l_2,l_3}(z_1,z_2)+
z_1^{l_1}\varphi^{k_1,k_2}_{k_1-l_1,l_1+l_2,l_3-l_1}(qz_1,z_2).
\label{Rd}
\ena
In general, the equality does not hold, however,
we will see that it does hold in a restricted range of the parameters.
Define the following sets:
\bea
&&R_U^{k_1,k_2}=\{(l_1,l_2,l_3)\mid
k_1\leq l_1+l_2-l_3\leq k_2\}\cap
P_U^{k_1,k_2},
\label{Rphi}
\\
&&R_V^{k_1,k_2}=\{(l_1,l_2,l_3)\mid
0\leq l_1+l_2-l_3\leq k_2-k_1\}
\cap P_V^{k_1,k_2},
\label{Rpsi}
\\
&&\bR_U^{k_1,k_2}=\{(l_1,l_2,l_3)\mid
0\leq l_1+l_2-l_3\leq k_2\}\cap P_U^{k_1,k_2}.
\label{Rphib}
\ena

\begin{thm}\label{thm:chrec}
The following recurrence relations hold.

If $(l_1,l_2,l_3)\in R_U^{k_1,k_2}$,  then
\bea
&&\varphi^{k_1,k_2}_{l_1,l_2,l_3}(z_1,z_2)
=\varphi^{k_1,k_2}_{l_1,l_2-1,\min(l_3,l_2-1)}(z_1,z_2)
+z_2^{l_2}\psi^{k_1,k_2}_{l_3,k_2-l_2,l_1}(q^{-1}z_1,qz_2)\,.
\label{TRa}
\ena
If $(l_1,l_2,l_3)\in R_V^{k_1,k_2}$, then
\bea
&&\psi^{k_1,k_2}_{l_1,l_2,l_3}(z_1,z_2)
=\psi^{k_1,k_2}_{l_1,l_2-1,\min(l_3,l_1+l_2-1)}(z_1,z_2)
+z_2^{l_2}\varphi^{k_1,k_2}_{l_3,k_2-l_2,l_1}(z_1,qz_2)\,.
\label{TRb}
\ena
If $(l_1,l_2,l_3)\in \bR_U^{k_1,k_2}$
and either $l_1+l_2-l_3\neq k_2$ or $l_3=0$, then
\bea
&&\varphi^{k_1,k_2}_{l_1,l_2,l_3}(z_1,z_2)
=\varphi^{k_1,k_2}_{l_1,l_2,l_3-1}(z_1,z_2)
\label{TRc}\\
&&\quad
+(q^{-1}z_1z_2)^{l_3}
\psi^{k_1,k_2}_{l_1-l_3,l_2-l_3,\min(k_1-l_3,l_1+l_2-2l_3)}(z_1,z_2),
\nonumber
\ena
If $(l_1,l_2,l_3)\in R_V^{k_1,k_2}$, then
\bea
&&\psi^{k_1,k_2}_{l_1,l_2,l_3}(z_1,z_2)
=\psi^{k_1,k_2}_{l_1-1,l_2,\min(l_3,l_1+l_2-1)}(z_1,z_2)
+
z_1^{l_1}\varphi^{k_1,k_2}_{k_1-l_1,l_1+l_2,l_3-l_1}(qz_1,z_2).
\label{TRd}
\ena
\end{thm}
The proof of Theorem \ref{thm:chrec}
will be completed in Section \ref{sec:5}.

We remark that, in \eqref{TRc}, the condition
$l_1+l_2-l_3\neq k_2$ is imposed so that
the parameters of the first term in the
right hand side stay within the region $\bR^{k_1,k_2}_U$.
In all other cases, the parameters
appearing in the right hand side belong to
the proper region ($\bR_U^{k_1,k_2}$ for
$\varphi^{k_1,k_2}_{l_1,l_2,l_3}$ and
$R_V^{k_1,k_2}$ for $\psi^{k_1,k_2}_{l_1,l_2,l_3}$).
In what follows we refer to \eqref{TRa}--\eqref{TRd} as
{\it short exact sequence (SES) recursion}.

\section{Vertex operators}

In this section we construct another family of $\nh$-modules
\be
\{(U_{VO})^{k_1,k_2}_{l_1,l_2,l_3} \mid (l_1,l_2,l_3)\in
\bR^{k_1,k_2}_U\},
\quad
\{(V_{VO})^{k_1,k_2}_{l_1,l_2,l_3} \mid (l_1,l_2,l_3)\in R^{k_1,k_2}_V\},
\en
as subspaces of tensor products of certain Fock modules.
These modules have the following properties.
\begin{enumerate}
\item $(U_{VO})^{k_1,k_2}_{l_1,l_2,l_3}$ is a quotient of
$U^{k_1,k_2}_{l_1,l_2,l_3}$, and
$(V_{VO})^{k_1,k_2}_{l_1,l_2,l_3}$ is a quotient of
$V^{k_1,k_2}_{l_1,l_2,l_3}$.
\item The characters
\be
&&(\varphi_{VO})^{k_1,k_2}_{l_1,l_2,l_3}(z_1,z_2)
=\ch_{z_1,z_2,q}(U_{VO})^{k_1,k_2}_{l_1,l_2,l_3},
\\
&&(\psi_{VO})^{k_1,k_2}_{l_1,l_2,l_3}(z_1,z_2)
=\ch_{z_1,z_2,q}(V_{VO})^{k_1,k_2}_{l_1,l_2,l_3}
\en
satisfy the following inequalities.
\end{enumerate}
If $(l_1,l_2,l_3)\in R^{k_1,k_2}_U$, then
\bea
(\varphi_{VO})^{k_1,k_2}_{l_1,l_2,l_3}(z_1,z_2)
\ge(\varphi_{VO})^{k_1,k_2}_{l_1,l_2-1,\min(l_3,l_2-1)}(z_1,z_2)+
z_2^{l_2}(\psi_{VO})^{k_1,k_2}_{l_3,k_2-l_2,l_1}(q^{-1}z_1,qz_2).
\label{a}
\ena
If $(l_1,l_2,l_3)\in R^{k_1,k_2}_V$, then
\bea
(\psi_{VO})^{k_1,k_2}_{l_1,l_2,l_3}(z_1,z_2)
\ge(\psi_{VO})^{k_1,k_2}_{l_1,l_2-1,\min(l_3,l_1+l_2-1)}(z_1,z_2)
+
z_2^{l_2}(\varphi_{VO})^{k_1,k_2}_{l_3,k_2-l_2,l_1}(z_1,qz_2).
\label{b}
\ena
If $(l_1,l_2,l_3)\in \bR^{k_1,k_2}_U$
and $l_1+l_2-l_3\neq0$ or $l_3=0$, then
\begin{multline}\label{c}
(\varphi_{VO})^{k_1,k_2}_{l_1,l_2,l_3}(z_1,z_2)
\ge(\varphi_{VO})^{k_1,k_2}_{l_1,l_2,l_3-1}(z_1,z_2)\\+
(q^{-1}z_1z_2)^{l_3}
(\psi_{VO})^{k_1,k_2}_{l_1-l_3,l_2-l_3,\min(k_1-l_3,l_1+l_2-2l_3)}(z_1,z_2).
\end{multline}
If $(l_1,l_2,l_3)\in R^{k_1,k_2}_V$, then
\begin{multline}\label{d}
(\psi_{VO})^{k_1,k_2}_{l_1,l_2,l_3}(z_1,z_2)
\ge(\psi_{VO})^{k_1,k_2}_{l_1-1,l_2,\min(l_3,l_1+l_2-1)}(z_1,z_2)
\\+
z_1^{l_1}(\varphi_{VO})^{k_1,k_2}_{k_1-l_1,l_1+l_2,l_3-l_1}(qz_1,z_2).
\end{multline}
These inequalities differ from \eqref{TRa}--\eqref{TRd}
by the change of the sign $=$ to the sign $\ge$.

Let us recall some constructions from the theory of
lattice vertex operator algebras (see \cite{D}, \cite{K2}).
Let $\mathfrak{h}$ be a two-dimensional complex vector space with a basis
$a$,$b$ and an inner product defined by the
$\mathfrak{sl}_3$ Cartan matrix:
\be
(a,a)=2,\ (b,b)=2,\ (a,b)=-1.
\en
Let
$$
\widehat\h=\h\T\C[t,t^{-1}]\oplus \C 1
$$
be the corresponding Heisenberg Lie algebra with the bracket
$$
[\al[i], \beta[j]]=i\delta_{i,-j}(\al,\beta)
\quad(\al,\beta\in\h),
$$
where $\alpha[i]=\alpha\otimes t^i$.
{}For $\al\in\h$ define the
Fock representation $\mathcal F_\al$ generated by a vector $\ket{\al}$ such that
$$
\beta[n] \ket{\al}=0,\ n>0, \qquad \beta[0] \ket{\al}=(\beta,\al) \ket{\al}.
$$
Set $L=\Z (2a+b)/3\oplus \Z (a+2b)/3$.
{}For $\al\in L$ we consider the corresponding vertex operators
acting on the direct sum of Fock spaces
$\bigoplus_{\al\in L} {\mathcal F}_\al$:
\be
\Gamma_\al(z)=S_\al z^{\al[0]} \exp(-\sum_{n<0} \frac{\al[n]}{n} z^{-n})
\exp(-\sum_{n>0} \frac{\al[n]}{n} z^{-n}),
\en
where $z^{\al[0]}$ acts on $\mathcal{F}_\beta$ by $z^{(\al,\beta)}$ and
the operator $S_\al$ is defined by
$$
S_\al \ket{\beta}=\ket{\al+\beta},
\quad
[S_\al,\beta[n]]=0
\quad (n\ne 0, \al,\beta\in\h).
$$
The Fourier decomposition is given by
$$\Gamma_\al(z)=\sum_{n\in\Z} \Gamma_\al[n] z^{-n-(\al,\al)/2}.$$
In particular,
\be
\Gamma_\al[-(\al,\al)/2-(\al,\beta)]\ket{\beta}=\ket{\al+\beta}.
\en

We need three vertex operators corresponding to the vectors $a$, $b$ and $c=a+b$.
We fix the notation
\begin{gather*}
a(z)=\Gamma_a(z),\quad b(z)=\Gamma_b(z),\quad c(z)=\Gamma_c(z),\\
a[n]=\Gamma_a[n], \quad b[n]=\Gamma_b[n],\quad c[n]=\Gamma_c[n].
\end{gather*}
The Frenkel-Kac construction for level $1$ modules (see \cite{FK})
defines the action of $\nh$ on $\oplus_{\alpha\in L}\mathcal F_\alpha$
via the homomorphism
$$e_{21}[n]\mapsto a[n], \quad e_{32}[n]\mapsto b[n],\quad e_{31}[n]\mapsto c[n].$$

Let $v_{m,n}$ be a  vacuum vector of the Fock module
$
\mathcal F_{-(\frac{2m+n}{3}+1)a-(\frac{2n+m}{3}+1)b}.
$
\begin{lem}
\label{comp}
We have
\be
&&e_{21}[i]v_{m,n}=0
\quad (i>m),\\
&&e_{32}[i]v_{m,n}=0
\quad (i>n),\\
&&e_{31}[i]v_{m,n}=0
\quad (i>m+n+1),\\
&&e_{21}[m]v_{m,n}=v_{m-2,n+1},\\
&&e_{32}[n]v_{m,n}=v_{m+1,n-2},\\
&&e_{31}[m+n+1]v_{m,n}=v_{m-1,n-1}.
\en
\end{lem}

We denote by $W_{m,n}\hookrightarrow \oplus_{\alpha\in L}\mathcal F_\alpha$
the cyclic $\nh$-module with the cyclic vector $v_{m,n}$.
The shift automorphism $T_{m,n}$ induces an isomorphism between
$W_{0,0}[m,n]$ and $W_{m,n}$.

We also need one-parameter analogues of the modules $W_{m,n}$.
{}Fix a one-dimensional space with a basis vector ${\bar b}$
and an inner product defined by $({\bar b},{\bar b})=2$.
The vertex operator $\bar b(z)=\Gamma_{\bar b}(z)$
acts in the direct sum of Fock
modules over the Heisenberg algebra $\C{\bar b}\otimes\C[t,t^{-1}]\oplus\C\cdot1$.
We set $W_n=\C[\{{\bar b}[i]\}_{i\in\Z}]\cdot v_n$ where $v_n=\ket{-(n+1)b/2}$,
and make it an $\nh$-module by letting $e_{32}[i]$ act by ${\bar b}[i]$ and
$e_{21}[i],e_{31}[i]$ by $0$.
Then we have
\be
e_{32}[i]v_n=0
\quad(i>n),\
e_{32}[n]v_n=v_{n-2}.
\en
Now we define the VO (vertex operator) versions
of the spaces $U_{l_1,l_2,l_3}^{k_1,k_2}$ and $V_{l_1,l_2,l_3}^{k_1,k_2}$,
utilizing the modules
\be
W_{0,0},~~W_{0,-1},~~W_{-1,0},~~W_{-1,-1},~~W_0,~~W_{-1}
\en
as building blocks. 

\begin{definition}\label{one}
Let $(l_1,l_2,l_3)\in \bR^{k_1,k_2}_U$.

If $(l_1,l_2,l_3)\in R^{k_1,k_2}_U$, then we define
\be
(U_{VO})^{k_1,k_2}_{l_1,l_2,l_3}\subset
W_{0,0}^{\otimes l_3}\otimes
W_{0,-1}^{\otimes (l_1-l_3)}\otimes
W_{-1,0}^{\otimes (k_1-l_1)}\otimes
W_0^{\otimes (l_1+l_2-l_3-k_1)} \otimes
W_{-1}^{\otimes (k_2-l_1-l_2+l_3)}
\en
to be the cyclic $\nh$-module with the cyclic vector
\be
w_1(l_1,l_2,l_3)=v_{0,0}^{\otimes l_3}\otimes
v_{0,-1}^{\otimes (l_1-l_3)}\otimes
v_{-1,0}^{\otimes (k_1-l_1)}\otimes
v_0^{\otimes (l_1+l_2-l_3-k_1)}\otimes
v_{-1}^{\otimes (k_2-l_1-l_2+l_3)}.
\en

If $l_1+l_2-l_3< k_1$, then we define
\be
(U_{VO})^{k_1,k_2}_{l_1,l_2,l_3}\subset
W_{0,0}^{\otimes l_3}\otimes
W_{0,-1}^{\otimes (l_1-l_3)}\otimes
W_{-1,0}^{\otimes (l_2-l_3)}\otimes
W_{-1,-1}^{\otimes (k_1-l_1-l_2+l_3)}\otimes
W_{-1}^{\otimes (k_2-k_1)}
\en
to be the cyclic $\nh$-module with the cyclic vector
\be
w_2(l_1,l_2,l_3)=
v_{0,0}^{\otimes l_3}\otimes
v_{0,-1}^{\otimes (l_1-l_3)}\otimes
v_{-1,0}^{\otimes (l_2-l_3)}\otimes
v_{-1,-1}^{\otimes (k_1-l_1-l_2+l_3)}\otimes
v_{-1}^{\otimes (k_2-k_1)}.
\en
\end{definition}

\begin{lem}\label{surj1}
Let $(l_1,l_2,l_3)\in \bR^{k_1,k_2}_U$.
Then there exists a surjective homomorphism
of $\nh$-modules
$$U^{k_1,k_2}_{l_1,l_2,l_3}\to (U_{VO})^{k_1,k_2}_{l_1,l_2,l_3}.$$
\end{lem}
\begin{proof}
It is sufficient to check that the defining relations of $U^{k_1,k_2}_{l_1,l_2,l_3}$
are satisfied in $(U_{VO})^{k_1,k_2}_{l_1,l_2,l_3}$. That follows from
Lemma $\ref{comp}$.
\end{proof}

\begin{definition}
If $(l_1,l_2,l_3)\in R^{k_1,k_2}_V$, then we define
$$
(V_{VO})^{k_1,k_2}_{l_1,l_2,l_3}\subset
W_{0,-1}^{\otimes l_1}\otimes
W_{-1,0}^{\otimes (l_3-l_1)}\otimes
W_{-1,-1}^{\otimes (k_1-l_3)}\otimes
W_0^{\otimes (l_1+l_2-l_3)}\otimes
W_{-1}^{\otimes (k_2-k_1-l_1-l_2+l_3)}
$$
to be the cyclic $\nh$-module  with the cyclic vector
\be
w_3(l_1,l_2,l_3)=v_{0,-1}^{\otimes l_1}\otimes
v_{-1,0}^{\otimes (l_3-l_1)}\otimes
v_{-1,-1}^{\otimes (k_1-l_3)}\otimes
v_0^{\otimes (l_1+l_2-l_3)}\otimes
v_{-1}^{\otimes (k_2-k_1-l_1-l_2+l_3)}.
\en
\end{definition}

\begin{lem}\label{surj2}
Let $(l_1,l_2,l_3)\in R^{k_1,k_2}_V$.
Then there exists a surjective homomorphism
of $\nh$-modules
$$V^{k_1,k_2}_{l_1,l_2,l_3}\to (V_{VO})^{k_1,k_2}_{l_1,l_2,l_3}.$$
\end{lem}
\begin{proof}
The lemma follows from Lemma $\ref{comp}$.
\end{proof}


\begin{thm}\label{mainbos}
There exist the following complexes
of $\nh$-modules 
which are exact in the first and third terms.

{}For $(l_1,l_2,l_3)\in R^{k_1,k_2}_U$, we have
\bea
0\to (V_{VO})^{k_1,k_2}_{l_3,k_2-l_2,l_1}[1,-1]
\buildrel{\iota_1}  
\over\longrightarrow (U_{VO})^{k_1,k_2}_{l_1,l_2,l_3}\to
(U_{VO})^{k_1,k_2}_{l_1,l_2-1,\min(l_3,l_2-1)}\to 0,
\label{am}
\ena
such that $\iota_1(T_{1,-1}(w_3(l_3,k_2-l_2,l_1)))=e_{32}^{l_2}[0]w_1(l_1,l_2,l_3)$.

{}For $(l_1,l_2,l_3)\in R^{k_1,k_2}_V$, we have
\bea
0\to (U_{VO})^{k_1,k_2}_{l_3,k_2-l_2,l_1}[0,-1]
\buildrel{\iota_2}
\over\longrightarrow
(V_{VO})^{k_1,k_2}_{l_1,l_2,l_3}\to
(V_{VO})^{k_1,k_2}_{l_1,l_2-1,\min(l_3,l_1+l_2-1)}\to 0,
\label{bm}
\ena
such that $\iota_2(T_{0,-1}(w_1(l_3,k_2-l_2,l_1)))=
e_{32}^{l_2}[0]w_3(l_1,l_2,l_3)$.

{}For $(l_1,l_2,l_3)\in \overline R^{k_1,k_2}_U$ with $l_1+l_2-l_3\neq k_2$ or
$l_3=0$,
\bea
0\to (V_{VO})^{k_1,k_2}_{l_1-l_3,l_2-l_3,\min(k_1-l_3,l_1+l_2-2l_3)}
\buildrel{\iota_3}
\over\longrightarrow
  (U_{VO})^{k_1,k_2}_{l_1,l_2,l_3}\to
(U_{VO})^{k_1,k_2}_{l_1,l_2,l_3-1}\to 0,
\label{cm}
\ena
such that $\iota_3(w_3(l_1-l_3,l_2-l_3,k_1-l_3))=
e_{31}^{l_3}[1]w_1(l_1,l_2,l_3)$ in the case of
$(l_1,l_2,l_3)\in R^{k_1,k_2}_U$ and
$\iota_3(w_3(l_1-l_3,l_2-l_3,l_1+l_2-2l_3))=
e_{31}^{l_3}[1]w_2(l_1,l_2,l_3)$
 otherwise.

{}For $(l_1,l_2,l_3)\in R^{k_1,k_2}_V$,
\bea \qquad
0\to (U_{VO})^{k_1,k_2}_{k_1-l_1,l_1+l_2,l_3-l_1}[-1,0]
\buildrel{\iota_4}
\over\longrightarrow
(V_{VO})^{k_1,k_2}_{l_1,l_2,l_3}\to
(V_{VO})^{k_1,k_2}_{l_1-1,l_2,\min(l_3,l_1+l_2-1)}\to 0,
\label{dm}
\ena
such that $\iota_4(T_{-1,0}(w_1(k_1-l_1,l_1+l_2,l_3-l_1)))=
e_{21}[0]^{l_1} w_3(l_1,l_2,l_3)$.

In these formulas, if one of the indices is negative, then
the corresponding term is understood as zero.
\end{thm}

\begin{cor}
Inequalities \eqref{a}, \eqref{b}, \eqref{c}
and \eqref{d} are satisfied.
\end{cor}

The rest of the section is devoted to the proof of Theorem \ref{mainbos}.
We start with the proof of the existence of the embeddings.

\begin{prop}
Let $(l_1,l_2,l_3)\in R^{k_1,k_2}_U$. Then
we have embeddings of $\nh$-modules
\be
&&(V_{VO})^{k_1,k_2}_{l_3,k_2-l_2,l_1}[1,-1]
\buildrel{\iota_1}
\over\longrightarrow(U_{VO})^{k_1,k_2}_{l_1,l_2,l_3},\\
&&(V_{VO})^{k_1,k_2}_{l_1-l_3,l_2-l_3,k_1-l_3}\buildrel{\iota_3}
\over\longrightarrow(U_{VO})^{k_1,k_2}_{l_1,l_2,l_3},
\en
such that $\iota_1(T_{1,-1}(w_3(l_3,k_2-l_2,l_1)))=e_{32}^{l_2}[0]w_1(l_1,l_2,l_3)$ and $\iota_3(w_3(l_1-l_3,l_2-l_3,k_1-l_3))=
e_{31}^{l_3}[1]w_1(l_1,l_2,l_3)$.

Let $(l_1,l_2,l_3)\in \bR^{k_1,k_2}_U$ and $l_1+l_2-l_3< k_1$.
Then we have an embedding of $\nh$-modules
\be
&&(V_{VO})^{k_1,k_2}_{l_1-l_3,l_2-l_3,l_1+l_2-2l_3}\buildrel{\iota_3}
\over\longrightarrow(U_{VO})^{k_1,k_2}_{l_1,l_2,l_3},
\en
such that $\iota_3(w_3(l_1-l_3,l_2-l_3,l_1+l_2-2l_3))=
e_{31}^{l_3}[1]w_2(l_1,l_2,l_3)$.

Let $(l_1,l_2,l_3)\in R^{k_1,k_2}_V$. Then
we have embeddings of $\nh$-modules
\be
&&(U_{VO})^{k_1,k_2}_{l_3,k_2-l_2,l_1}[0,-1]\buildrel{\iota_2}
\over\longrightarrow(V_{VO})^{k_1,k_2}_{l_1,l_2,l_3},\\
&&(U_{VO})^{k_1,k_2}_{k_1-l_1,l_1+l_2,l_3-l_1}[-1,0]
\buildrel{\iota_4}
\over\longrightarrow(V_{VO})^{k_1,k_2}_{l_1,l_2,l_3},
\en
such that $\iota_2(T_{0,-1}(w_1(l_3,k_2-l_2,l_1)))=
e_{32}^{l_2}[0]w_3(l_1,l_2,l_3)$ and
$\iota_4(T_{-1,0}(w_1(k_1-l_1,l_1+l_2,l_3-l_1)))=
e_{21}[0]^{l_1} w_3(l_1,l_2,l_3)$.
\end{prop}

\begin{proof}
We prove the first embedding. The proof of the other embeddings
is similar.
By Lemma \ref{comp} we have
\be
&&e_{32}[0]^{l_2}(
v_{0,0}^{\otimes l_3}\otimes
v_{0,-1}^{\otimes (l_1-l_3)}\otimes
v_{-1,0}^{\otimes (k_1-l_1)}\otimes
v_0^{\otimes (l_1+l_2-l_3-k_1)}\otimes
v_{-1}^{\otimes (k_2-l_1-l_2+l_3)})\\=
&&
v_{1,-2}^{\otimes l_3}\otimes
v_{0,-1}^{\otimes (l_1-l_3)}\otimes
v_{0,-2}^{\otimes (k_1-l_1)}\otimes
v_{-2}^{\otimes (l_1+l_2-l_3-k_1)}\otimes
v_{-1}^{\otimes (k_2-l_1-l_2+l_3)}.
\en
The vector in the second line 
coincides with the cyclic vector of $(V_{VO})^{k_1,k_2}_{l_3,k_2-l_2,l_1}[1,-1]$.
\end{proof}

We prepare some facts about the isomorphisms between
the modules
$U^{k_1,k_2}_{l_1,l_2,l_3}$
and
$(U_{VO})^{k_1,k_2}_{l_1,l_2,l_3}$.

\begin{prop}\label{GORDON}
Let $l_3=\min(l_1,l_2)$. Then we have the isomorphism
\be
U^{k_1,k_2}_{l_1,l_2,l_3}\simeq
(U_{VO})^{k_1,k_2}_{l_1,l_2,l_3},
\en
and the corresponding character is  given by the fermionic formula
\begin{equation}\label{fermform}
\phi^{k_1,k_2}_{l_1,l_2,l_3}(z_1,z_2)=
(\phi_{VO})^{k_1,k_2}_{l_1,l_2,l_3}(z_1,z_2)=
\sum_{\genfrac{}{}{0pt}{}{m_1,\ldots,m_{k_1}\ge 0}{n_1,\ldots,n_{k_2}\ge 0}}
\frac{q^{Q(m,n)-\sum_{i=1}^{k_1}\min(l_1,i)m_i-\sum_{i=1}^{k_2}\min(l_2,i)n_i}}
{(q)_{m_1}\cdots(q)_{m_{k_1}}(q)_{n_1}\cdots(q)_{n_{k_2}}},
\end{equation}
where
$$
Q(m,n)=
\sum_{i,j=1}^{k_1}\min(i,j)m_im_j
+\sum_{i,j=1}^{k_2}\min(i,j)n_in_j-\sum_{i=1}^{k_1}
\sum_{j=1}^{k_2}\min(i,j)m_in_j.
$$
\end{prop}
\begin{proof}
Because of the existence of the surjection of $\nh$-modules
\begin{equation}\label{surj}
U^{k_1,k_2}_{l_1,l_2,l_3}\to
(U_{VO})^{k_1,k_2}_{l_1,l_2,l_3},
\end{equation}
it is sufficient to prove $(\ref{fermform})$.

{}For the rest of the proof, we assume $l_3=\min(l_1,l_2)$.
In this case the defining relations of
$U_{l_1,l_2,l_3}^{k_1,k_2}$ reduce to
$$
e_{21}(z)^{k_1+1}=0, \quad
e_{32}(z)^{k_2+1}=0, \quad
e_{21}[0]^{l_1+1}v=0,\quad
e_{32}[0]^{l_2+1}v=0,
$$
where $v$ is a cyclic vector of $U_{l_1,l_2,l_3}^{k_1,k_2}$.
Denote the fermionic formula in the right hand side of \eqref{fermform}
by $F^{k_1,k_2}_{l_1,l_2}(z_1,z_2)$.
By using the Gordon filtration technique and
repeating the proof in \cite{AKS} in the case $k_1<k_2$,
we obtain an upper estimate
\be
F^{k_1,k_2}_{l_1,l_2}(z_1,z_2)
\ge
\varphi^{k_1,k_2}_{l_1,l_2,l_3}(z_1,z_2).
\en
{}From surjection \eqref{surj} we know that
\be
\varphi^{k_1,k_2}_{l_1,l_2,l_3}(z_1,z_2)
\ge (\varphi_{VO})^{k_1,k_2}_{l_1,l_2,l_3}(z_1,z_2).
\en
To finish the proof it remains to show the inequality
\bea
(\varphi_{VO})^{k_1,k_2}_{l_1,l_2,l_3}(z_1,z_2)\ge
F^{k_1,k_2}_{l_1,l_2}(z_1,z_2).
\label{es2}
\ena

The space $(U_{VO})^{k_1,k_2}_{l_1,l_2,l_3}$
is generated by the modes of
$$
e_{21}(z)=\sum_{i=1}^{k_1} a_i(z)
\quad \text{ and }\quad
e_{32}(z)=\sum_{j=1}^{k_2} b_j(z)
$$
from the vector $w_1(l_1,l_2,l_3)$ or $w_2(l_1,l_2,l_3)$
(see Definition \ref{one}).
The vertex operators $a_i(z)$, $b_j(z)$ correspond
to the vectors
$a_i$, $1\le i\le k_1$, and $b_j$, $1\le j\le k_2$, respectively, with
the scalar products given by
$$
(a_{i_1},a_{i_2})=2\delta_{i_1,i_2},
\quad(b_{j_1},b_{j_2})=2\delta_{j_1,j_2}, \quad
(a_i, b_j)=-\delta_{i,j}.
$$
{}For a non-zero complex number $\varepsilon$, set
$$
e^\ve_{21}(z)=\sum_{i=1}^{k_1} \ve^i a_i(z),\quad
e^\ve_{32}(z)=\sum_{j=1}^{k_2} \ve^j b_j(z).
$$
The currents $a_i(z)$ are mutually commutative
and $a_i(z)^2=0$. The same holds for the currents $b_j(z)$.
Hence we have
\begin{gather*}
\lim_{\ve\to 0} \ve^{-i(i+1)/2} (e^\ve_{21}(z))^i=a_1(z)\dots a_i(z),\
\lim_{\ve\to 0} \ve^{-i(i+1)/2} (e^\ve_{32}(z))^i=b_1(z)\dots b_i(z).
\end{gather*}
Consider the subspace
generated from the vector $w_1(l_1,l_2,l_3)$ or $w_2(l_1,l_2,l_3)$ by the
modes of the operators
$$
\mathbf{a}_i(z)=a_1(z)\dots a_i(z)\quad (1\le i\le k_1),
\qquad
\mathbf{b}_j(z)=b_1(z)\dots b_j(z)\quad (1\le j\le k_2),
$$
and let $(\overline{\varphi}_{VO})^{k_1,k_2}_{l_1,l_2,l_3}(z_1,z_2)$ be its
character.
Clearly, we have
\be
(\varphi_{VO})^{k_1,k_2}_{l_1,l_2,l_3}(z_1,z_2)\ge
(\overline{\varphi}_{VO})^{k_1,k_2}_{l_1,l_2,l_3}(z_1,z_2).
\en
The operator $\mathbf{a}_i(z)$ is a vertex operator corresponding to the vector
$\mathbf{a}_i=a_1+\dots +a_i$, and
$\mathbf{b}_j(z)$ is a vertex operator corresponding to the vector
$\mathbf{b}_j=b_1+\dots + b_j$.
The scalar products of these vectors are given by
$$
(\mathbf{a}_{i_1},\mathbf{a}_{i_2})=2\min(i_1, i_2),\ 
(\mathbf{b}_{j_1},\mathbf{b}_{j_2})=2\min(j_1, j_2),\
(\mathbf{a}_i,\mathbf{b}_j)=-\min(i, j).
$$
Using the  standard technique (see for example \cite{FJMMT}),
we obtain
\be
(\overline{\varphi}_{VO})^{k_1,k_2}_{l_1,l_2,l_3}(z_1,z_2)
=F^{k_1,k_2}_{l_1,l_2}(z_1,z_2).
\en
Inequality \eqref{es2} follows.

The proposition is proved.
\end{proof}

\begin{cor}\label{int}
We have the isomorphisms
\begin{gather}
W_0\simeq U^{0,1}_{0,1,0},\ W_{-1}\simeq U^{0,1}_{0,0,0},\notag
\\
W_{0,0}\simeq U^{1,1}_{1,1,1},\
W_{0,-1}\simeq U^{1,1}_{1,0,0},\
W_{-1,0}\simeq U^{1,1}_{0,1,0},\
W_{-1,-1}\simeq U^{1,1}_{0,0,0}.\notag 
\end{gather}
\end{cor}

Next, we prove the surjections in Theorem \ref{mainbos}
and show that all sequences are in fact complexes.

\begin{prop}\label{SURJ1}
Let $(l_1,l_2,l_3)\in R^{k_1,k_2}_U$ and $l_2>0$.
Then there exists a surjective homomorphism of
$\nh$-modules
\be
(U_{VO})^{k_1,k_2}_{l_1,l_2,l_3}/
\langle e_{32}[0]^{l_2}\rangle\rightarrow
(U_{VO})^{k_1,k_2}_{l_1,l_2-1,\min(l_3,l_2-1)}
\en
such that, if $l_1+l_2-1-\min(l_3,l_2-1)\ge k_1$ then
$$
w_1(l_1,l_2,l_3)\mapsto w_1(l_1,l_2-1,\min(l_3,l_2-1)),
$$ and if $l_1+l_2-1-\min(l_3,l_2-1)= k_1-1$ then
$$
w_1(l_1,l_2,l_3)\mapsto w_2(l_1,l_2-1,\min(l_3,l_2-1)).
$$
\end{prop}

\begin{proof}
We consider the case $l_2>l_3$ and $l_1+l_2-1-l_3<k_1$.
The other cases are similar.
Replacing a factor $v_{-1,0}$ in $w_1(l_1,l_2,l_3)$
by the factor $v_{-1,-1}$,
we obtain $w_2(l_1,l_2-1,l_3)$. We have an obvious surjective
homomorphism of $\nh$-modules (see Lemma \ref{int})
$$
W_{-1,0}\to W_{-1,-1},\quad v_{-1,0}\mapsto v_{-1,-1}.
$$
Therefore, we obtain a surjective homomorphism
$$
(U_{VO})^{k_1,k_2}_{l_1,l_2,l_3}\to
(U_{VO})^{k_1,k_2}_{l_1,l_2-1,l_3},\quad
w_1(l_1,l_2,l_3)\mapsto w_2(l_1,l_2-1,l_3).
$$
In addition, the vector $e_{32}[0]^{l_2}w_1(l_1,l_2,l_3)$ maps to zero,
because $$e_{32}[0]^{l_2}w_2(l_1,l_2-1,l_3)=0.$$
The proposition is proved.
\end{proof}

\begin{cor}
Sequence (\ref{am}) is exact in the first and third terms.
\end{cor}

\begin{prop}
Suppose $(l_1,l_2,l_3)\in R^{k_1,k_2}_U$,
$l_1+l_2-l_3<k_2$ and $l_3>0$. Then we
have a surjective homomorphism of $\nh$-modules
\begin{gather}\label{UL3}
(U_{VO})^{k_1,k_2}_{l_1,l_2,l_3}/
\langle e_{31}[1]^{l_3}\rangle
\rightarrow
(U_{VO})^{k_1,k_2}_{l_1,l_2,l_3-1},\\ \nonumber
w_1(l_1,l_2,l_3)\mapsto w_1(l_1,l_2,l_3-1).
\end{gather}
\end{prop}

\begin{proof}
Replacing a factor $v_{0,0}\otimes v_{-1}$ in
$w_1(l_1,l_2,l_3)$
by the factor  $v_{0,-1}\T v_0$, we obtain
$w_1(l_1,l_2,l_3-1)$.
By Definition \ref{one},
\begin{equation*}
U(\nh)\cdot (v_{0,0}\T v_{-1})=(U_{VO})^{1,2}_{1,1,1},\quad
U(\nh)\cdot (v_{0,-1}\T v_0)=(U_{VO})^{1,2}_{1,1,0}.
\end{equation*}
Therefore to construct (\ref{UL3}) it is sufficient to construct
a homomorphism
$$
(U_{VO})^{1,2}_{1,1,1}\to (U_{VO})^{1,2}_{1,1,0},\quad
v_{0,0}\T v_{-1}\mapsto v_{0,-1}\T v_0.
$$
Using Proposition \ref{GORDON}, we have
$$(U_{VO})^{1,2}_{1,1,1}\simeq U^{1,2}_{1,1,1}\to U^{1,2}_{1,1,0}\to
(U_{VO})^{1,2}_{1,1,0}.$$
The proposition follows from $$e_{31}[1]^{l_3}w_1(l_1,l_2,l_3-1)=0.$$
\end{proof}

\begin{prop}
Suppose $(l_1,l_2,l_3)\in \bR^{k_1,k_2}_U$,
$l_1+l_2-l_3<k_1$ and $l_3>0$.
Then we have a surjection of $\nh$-modules
\begin{gather*}
(U_{VO})^{k_1,k_2}_{l_1,l_2,l_3}/
\langle e_{31}[1]^{l_3}\rangle
\rightarrow
(U_{VO})^{k_1,k_2}_{l_1,l_2,l_3-1},\\
w_2(l_1,l_2,l_3)\mapsto w_2(l_1,l_2,l_3-1).
\end{gather*}
\end{prop}
\begin{proof}
Replacing a factor $v_{0,0}\T v_{-1,-1}$ in $w_2(l_1,l_2,l_3)$ by
the factor $v_{0,-1}\T v_{-1,0}$, we obtain $w_2(l_1,l_2,l_3-1)$.
We have
$$U(\nh)\cdot v_{0,0}\T v_{-1,-1}= (U_{VO})^{2,2}_{1,1,1},\quad
U(\nh)\cdot v_{0,-1}\T v_{-1,0}= (U_{VO})^{2,2}_{1,1,0}.$$
We have a surjective homomorphism
$(U_{VO})^{2,2}_{1,1,1}\to (U_{VO})^{2,2}_{1,1,0}$, which is the composition
$$(U_{VO})^{2,2}_{1,1,1}\simeq U^{2,2}_{1,1,1}\to U^{2,2}_{1,1,0} \to (U_{VO})^{2,2}_{1,1,0}.$$
Therefore we obtain a surjection
$$
(U_{VO})^{k_1,k_2}_{l_1,l_2,l_3}\to (U_{VO})^{k_1,k_2}_{l_1,l_2,l_3-1}.
$$
The proposition follows from $$e_{31}[1]^{l_3}w(l_1,l_2,l_3-1)=0.$$
\end{proof}

\begin{cor}
Sequence (\ref{cm}) is exact in the first and third terms.
\end{cor}

\begin{prop}
Let $(l_1,l_2,l_3)\in R^{k_1,k_2}_V$ and $l_2>0$. Then
there exists a surjective homomorphism of $\nh$-modules
\begin{gather*}
(V_{VO})^{k_1,k_2}_{l_1,l_2,l_3}/
\langle e_{32}[0]^{l_2}\rangle
\rightarrow
(V_{VO})^{k_1,k_2}_{l_1,l_2-1,\min(l_3, l_1+l_2-1)},\\
w_3(l_1,l_2,l_3)\mapsto w_3(l_1,l_2-1,\min(l_3,l_1+l_2-1)).
\end{gather*}
\end{prop}
\begin{proof}
The proof is similar to the proof of Proposition \ref{SURJ1}.
\end{proof}

\begin{cor}
Sequence (\ref{bm}) is exact in the first and third terms.
\end{cor}

\begin{prop}
Let $(l_1,l_2,l_3)\in R^{k_1,k_2}_V$ and $l_1>0$. Then
there exists a surjective homomorphism of $\nh$-modules
\begin{gather*}
(V_{VO})^{k_1,k_2}_{l_1,l_2,l_3}/
\langle e_{21}[0]^{l_1}\rangle
\rightarrow
(V_{VO})^{k_1,k_2}_{l_1-1,l_2,\min(l_3, l_1+l_2-1)},\\
w_3(l_1,l_2,l_3)\mapsto w_3(l_1-1,l_2,\min(l_3,l_1+l_2-1)).
\end{gather*}
\end{prop}
\begin{proof}
The proof is done similarly to the proof of other cases with the help of
the surjective homomorphism
\begin{equation}\label{last}
(U_{VO})^{1,2}_{1,1,0}\to (U_{VO})^{1,2}_{0,1,0},\
v_{0,-1}\T v_0\mapsto v_{-1,0}\T v_{-1},
\end{equation}
which we construct below.

{}First, we show that
\begin{equation}\label{eq}
(U_{VO})^{1,2}_{1,1,0}\simeq U^{1,2}_{1,1,0}.
\end{equation}
By \eqref{Ra}, we have the inequality
\begin{equation*}
\varphi^{1,2}_{1,1,0}(z_1,z_2)
\le \varphi^{1,2}_{1,0,0}(z_1,z_2)+z_2\psi^{1,2}_{0,1,1}(q^{-1}z_1,qz_2).
\end{equation*}
By the definition we have an isomorphism
$$
V^{1,2}_{0,1,1}\simeq U^{1,2}_{0,1,0},
$$
Therefore
\begin{equation*}
\varphi^{1,2}_{1,1,0}(z_1,z_2)
\le \varphi^{1,2}_{1,0,0}(z_1,z_2)
+z_2\varphi^{1,2}_{0,1,0}(q^{-1}z_1,qz_2).
\end{equation*}
By Proposition $\ref{SURJ1}$
and $(V_{VO})^{1,2}_{0,1,1}\simeq (U_{VO})^{1,2}_{0,1,0}\subset W_{-1,0}\otimes W_{-1}$,
we have the inequality
\be
(\varphi_{VO})^{1,2}_{1,1,0}(z_1,z_2)\ge
(\varphi_{VO})^{1,2}_{1,0,0}(z_1,z_2)
+z_2(\varphi_{VO})^{1,2}_{0,1,0}(q^{-1}z_1,qz_2).
\en
By Proposition $\ref{GORDON}$, we obtain the following diagram:
\be
\begin{matrix}
\varphi^{1,2}_{1,1,0}(z_1,z_2)&\leq&z_2\,
\varphi^{1,2}_{0,1,0}(q^{-1}z_1,qz_2)&+&
\varphi^{1,2}_{1,0,0}(z_1,z_2)\\
{\buildrel\phantom{c}\over\vee}\raise1pt\hbox{$\scriptstyle|$}
&&\scriptstyle||&&\scriptstyle||\\[2pt]
(\varphi_{VO})^{1,2}_{1,1,0}(z_1,z_2)&\geq&
z_2\,(\varphi_{VO})^{1,2}_{0,1,0}(q^{-1}z_1,qz_2)&+&
(\varphi_{VO})^{1,2}_{1,0,0}(z_1,z_2)
\end{matrix}
\en
{}From here we obtain isomorphism (\ref{eq}).

Map (\ref{last}) is a composition of
the following mappings:
$$
(U_{VO})^{1,2}_{1,1,0}\simeq U^{1,2}_{1,1,0}\to U^{1,2}_{0,1,0}
\to (U_{VO})^{1,2}_{0,1,0}.
$$
\end{proof}

\begin{cor}
Sequence (\ref{dm}) is exact in the first and third terms.
\end{cor}

Theorem \ref{mainbos} is proved.
\section{Uniqueness of the solution to the SES-recursion}\label{sec:5}
\subsection{The main case of the recursion}

Recall the definition of the regions of the parameters $R_U^{k_1,k_2},R_V^{k_1,k_2},\overline R_U^{k_1,k_2}$ (see \eqref{Rphi}--\eqref{Rphib}).

Let $\bar{\varphi}^{k_1,k_2}_{l_1,l_2,l_3}(z_1,z_2)$
($(l_1,l_2,l_3)\in \bR_U^{k_1,k_2}$)
and $\bar{\psi}^{k_1,k_2}_{l_1,l_2,l_3}(z_1,z_2)$
($(l_1,l_2,l_3)\in {R}_V^{k_1,k_2}$) be formal power series in variables
$z_1,z_2$ whose coefficients are Laurent power series in $q$.

We use the following convention.
A series with a negative index is understood to be zero.
If $l_3>\min(l_2,l_3)$, then
$\bar\varphi^{k_1,k_2}_{l_1,l_2,l_3}(z_1,z_2):=
\bar\varphi^{k_1,k_2}_{l_1,l_2,\min(l_1,l_2)}(z_1,z_2)$.
If $l_3>l_1+l_2$, then
$\bar\psi^{k_1,k_2}_{l_1,l_2,l_3}(z_1,z_2)
:=\bar\psi^{k_1,k_2}_{l_1,l_2,l_1+l_2}(z_1,z_2)$.

Assume that the following inequalities hold:

If $(l_1,l_2,l_3)\in {R}_U^{k_1,k_2}$, then
\bea
{\bar\varphi}^{k_1,k_2}_{l_1,l_2,l_3}(z_1,z_2)\label{M1}
\leq{\bar \varphi}^{k_1,k_2}_{l_1,l_2-1,l_3}(z_1,z_2)+
z_2^{l_2}{\bar\psi}^{k_1,k_2}_{l_3,k_2-l_2,l_1}(q^{-1}z_1,qz_2).
\ena
If $(l_1,l_2,l_3)\in {R}_V^{k_1,k_2}$, then
\bea
{\bar\psi}^{k_1,k_2}_{l_1,l_2,l_3}(z_1,z_2)\label{M2}
\leq {\bar\psi}^{k_1,k_2}_{l_1,l_2-1,l_3}(z_1,z_2)+
z_2^{l_2}{\bar\varphi}^{k_1,k_2}_{l_3,k_2-l_2,l_1}(z_1,qz_2).
\ena
If $(l_1,l_2,l_3)\in \bR_U^{k_1,k_2}$ and either
$l_1+l_2-l_3\neq k_2$ or $l_3=0$, then
\bea
\lefteqn{ {\bar\varphi}^{k_1,k_2}_{l_1,l_2,l_3}(z_1,z_2)\label{M3}} \notag\\
&&\hspace{40pt}
\leq{\bar\varphi}^{k_1,k_2}_{l_1,l_2,l_3-1}(z_1,z_2)+(q^{-1}z_1z_2)^{l_3}
{\bar\psi}^{k_1,k_2}_{l_1-l_3,l_2-l_3,k_1-l_3}(z_1,z_2).
\ena
If $(l_1,l_2,l_3)\in {R}_V^{k_1,k_2}$, then
\bea
 {\bar\psi}^{k_1,k_2}_{l_1,l_2,l_3}(z_1,z_2)\label{M4}
\leq{\bar\psi}^{k_1,k_2}_{l_1-1,l_2,l_3}(z_1,z_2)+
z_1^{l_1}{\bar\varphi}^{k_1,k_2}_{k_1-l_1,l_1+l_2,l_3-l_1}(qz_1,z_2).
\ena

We call formal power series of the forms
\be
p(z_1,z_2){\bar\varphi}^{k_1,k_2}_{l_1,l_2,l_3}(z_1,z_2)
\quad((l_1,l_2,l_3)\in \bR_U^{k_1,k_2}),
\quad
p(z_1,z_2){\bar\psi}^{k_1,k_2}_{l_1,l_2,l_3}(z_1,z_2)
\quad((l_1,l_2,l_3)\in {R}_V^{k_1,k_2})
\en
{\it higher degree series} if
$p(z_1,z_2)$ is a polynomial in $z_1,z_2$ whose coefficients are
Laurent polynomials in $q$ and if $p(0,0)=0$.

Let $F(z_1,z_2)$ and $G(z_1,z_2)$ be formal power series in variables
$z_1,z_2$. We write $F(z_1,z_2){\leq_*} G(z_1,z_2)$ if there exist
higher degree series $H_1(z_1,z_2),\dots,H_s(z_1,z_2)$ such that \newline
$F(z_1,z_2)\leq G(z_1,z_2)+\sum_{i=1}^s H_i(z_1,z_2)$.

\begin{lem}\label{loop}
Under the assumptions above,
let $F(z_1,z_2)$ be either
${\bar\varphi}^{k_1,k_2}_{l_1,l_2,l_3}(z_1,z_2)$, for some $(l_1,l_2,l_3)\in
  \bR_U^{k_1,k_2}$, or ${\bar
\psi}^{k_1,k_2}_{l_1,l_2,l_3}(z_1,z_2)$, for some $(l_1,l_2,l_3)\in
{R}_V^{k_1,k_2}$. Then there exist $m_1,m_2\in\Z_{\geq 1}$ such that
$F(z_1,z_2){\leq_*}{\bar
\psi}_{0,0,0}^{k_1,k_2}(q^{m_1}z_1,q^{m_2}z_2)$.
\end{lem}
\begin{proof}
{}First, consider the case $F(z_1,z_2)={\bar \psi}_{0,0,0}^{k_1,k_2}(z_1,z_2)$.
We use the given inequalities as follows:
\be
{\bar\psi}_{0,0,0}^{k_1,k_2}(z_1,z_2){\leq_*}
{\bar\varphi}_{0,k_2,0}^{k_1,k_2}(z_1,qz_2){\leq_*}
{\bar\psi}_{0,k_2,k_1}^{k_1,k_2}(z_1,qz_2){\leq_*}
{\bar\varphi}_{k_1,k_2,k_1}^{k_1,k_2}(qz_1,qz_2).
\en
Here we used (\ref{M2}) then (\ref{M3}) then (\ref{M4}).
Then we use inequality (\ref{M1}) $k_2$ times to obtain:
\be
{\bar\varphi}_{k_1,k_2,k_1}^{k_1,k_2}(qz_1,qz_2){\leq_*}
{\bar\varphi}_{k_1,0,0}^{k_1,k_2}(qz_1,qz_2).
\en
{}Finally, using (\ref{M3}) followed by the $k_1$ applications of (\ref{M4}),
we obtain:
\be
{\bar\varphi}_{k_1,0,0}^{k_1,k_2}(qz_1,qz_2){\leq_*}
{\bar\psi}_{k_1,0,k_1}^{k_1,k_2}(qz_1,qz_2){\leq_*}
{\bar\psi}_{0,0,0}^{k_1,k_2}(qz_1,qz_2).
\en
Combining, we obtain: ${\bar\psi}_{0,0,0}^{k_1,k_2}(z_1,z_2){\leq_*}
{\bar\psi}_{0,0,0}^{k_1,k_2}(qz_1,qz_2)$.

Next, consider the case $F(z_1,z_2)={\bar
\varphi}_{l_1,l_2,l_3}^{k_1,k_2}(z_1,z_2)$, $(l_1,l_2,l_3)\in
{R}_U^{k_1,k_2}$.

We start using inequality (\ref{M1}) followed by (\ref{M3}). We repeat
this step $l_3$ times. Then we apply (\ref{M3})
one more time. We get
\be
{\bar \varphi}_{l_1,l_2,l_3}^{k_1,k_2}(z_1,z_2){\leq_*}
{\bar\varphi}_{l_1,l_2-l_3,0}^{k_1,k_2}(z_1,z_2){\leq_*}
{\bar\psi}_{l_1,l_2-l_3,k_1}^{k_1,k_2}(z_1,z_2).
\en

Then we use the $l_1$ applications of (\ref{M4}) and after that
the $l_2-l_3$ applications of (\ref{M2}). We obtain
\be
{\bar\psi}_{l_1,l_2-l_3,k_1}^{k_1,k_2}(z_1,z_2){\leq_*}
{\bar\psi}_{0,l_2-l_3,k_1}^{k_1,k_2}(z_1,z_2){\leq_*}
{\bar\psi}_{0,0,0}^{k_1,k_2}(z_1,z_2).
\en
Combining, we obtain
\be
{\bar
\varphi}_{l_1,l_2,l_3}^{k_1,k_2}(z_1,z_2){\leq_*}
{\bar\psi}_{0,0,0}^{k_1,k_2}(z_1,z_2){\leq_*}
{\bar\psi}_{0,0,0}^{k_1,k_2}(qz_1,qz_2).
\en

\medskip

Next, we consider the case $F(z_1,z_2)={\bar
\varphi}_{l_1,l_2,l_3}^{k_1,k_2}(z_1,z_2)$, $(l_1,l_2,l_3)\in
\bR_U^{k_1,k_2}$ and $(l_1,l_2,l_3)\not \in
{R}_U^{k_1,k_2}$.

Let $l=\min(l_3,k_1-l_1-l_2+l_3)$. We use $l$ applications of
$(\ref{M3})$ and obtain
\be
{\bar\varphi}_{l_1,l_2,l_3}^{k_1,k_2}(z_1,z_2){\leq_*}
{\bar\varphi}_{l_1,l_2,l_3-l}^{k_1,k_2}(z_1,z_2).
\en
We have either $(l_1,l_2,l_3-l)\in
{R}_U^{k_1,k_2}$, or $l_3-l=0$.
In the first case we are reduced to the situation treated
above. In the second case, we use $(\ref{M3})$
to obtain
\be
{\bar\varphi}_{l_1,l_2,0}^{k_1,k_2}(z_1,z_2)
{\leq_*}
{\bar\psi}_{l_1,l_2,0}^{k_1,k_2}(z_1,z_2),
\en
which again is the situation discussed above.

\medskip

Last, consider the case $F(z_1,z_2)={\bar
\psi}_{l_1,l_2,l_3}^{k_1,k_2}(z_1,z_2)$, $(l_1,l_2,l_3)\in
{R}_V^{k_1,k_2}$. We use (\ref{M2}) $l_2+1$ times and obtain:
\be
{\bar\psi}_{l_1,l_2,l_3}^{k_1,k_2}(z_1,z_2){\leq_*}
{\bar\psi}_{l_1,0,l_1}^{k_1,k_2}(z_1,z_2){\leq_*}
{\bar\varphi}_{l_1,k_2,l_1}^{k_1,k_2}(z_1,qz_2).
\en
Therefore we are again reduced to the previous cases and the lemma is proved.
\end{proof}

\begin{cor}\label{cor}
Let $\bar{\varphi}^{k_1,k_2}_{l_1,l_2,l_3}$, $(l_1,l_2,l_3)\in 
\bR_U^{k_1,k_2}$, and $\bar{\psi}^{k_1,k_2}_{l_1,l_2,l_3}$,
$(l_1,l_2,l_3)\in {R}_V^{k_1,k_2}$, be formal power series in variables
$z_1,z_2$, such that inequalities (\ref{M1})-(\ref{M4})
are satisfied.

Assume that all coefficients of the power series are Laurent
power series in $q$ with non-negative integer coefficients. Assume also
that the formal power series $\bar{\psi}^{k_1,k_2}_{0,0,0}$ has no
constant term.

Then all these formal power series are identically zero.
\end{cor}
\begin{proof}
Suppose the contrary, and
let $n_1,n_2$ be non-negative integers
such that one of the series has
a non-trivial coefficient of $z_1^{n_1}z_2^{n_2}$
and all coefficients
of $z_1^{\al_1}z_2^{\al_2}$ of all power series are zero if
either $\al_1<n_1$ or $\al_2<n_2$.
By the assumption and Lemma \ref{loop},
we have $n_1+n_2>0$.

Let $n_3$ be the smallest possible
integer such that the coefficient of
$z_1^{n_1}z_2^{n_2}q^{n_3}$ is non-zero in one of the series
$F(z_1,z_2)$. By our assumption, this coefficient is positive.

By Lemma \ref{loop}, we have $F(z_1,z_2)\leq \sum_{i=1}^s
H_i(z_1,z_2)+ \bar{\psi}^{k_1,k_2}_{0,0,0}(z_1q^{m_1},z_2q^{m_2})$,
where $H_i$ are higher degree series and $m_1,m_2\geq 1$. Clearly, the
coefficient of $z_1^{n_1}z_2^{n_2}q^{n_3}$  is zero on the right hand side of
this inequality and positive on the left hand side,
which is a contradiction.
\end{proof}

\subsection{All inequalities are equalities}
Recall that we have a set of $\hat{\mathfrak{n}}$-modules
$(U_{VO})^{k_1,k_2}_{l_1,l_2,l_3}$, $U^{k_1,k_2}_{l_1,l_2,l_3}$
for $(l_1,l_2,l_3)\in {\bR}_{U}^{k_1,k_2}$ and
$(V_{VO})^{k_1,k_2}_{l_1,l_2,l_3}$, $V^{k_1,k_2}_{l_1,l_2,l_3}$
for $(l_1,l_2,l_3)\in {R}_{V}^{k_1,k_2}$,
whose characters we denote by
$(\varphi_{VO})^{k_1,k_2}_{l_1,l_2,l_3}(z_1,z_2)$,
$\varphi^{k_1,k_2}_{l_1,l_2,l_3}(z_1,z_2)$ and
$(\psi_{VO})^{k_1,k_2}_{l_1,l_2,l_3}(z_1,z_2)$,
$\psi^{k_1,k_2}_{l_1,l_2,l_3}(z_1,z_2)$, respectively.

\begin{thm}\label{thm:equalities}
We have
\bea
(\varphi_{VO})^{k_1,k_2}_{l_1,l_2,l_3}(z_1,z_2)&=
&\varphi^{k_1,k_2}_{l_1,l_2,l_3}(z_1,z_2)
\qquad  ((l_1,l_2,l_3)\in\bR_U^{k_1,k_2}),
\notag
\\
(\psi_{VO})^{k_1,k_2}_{l_1,l_2,l_3}(z_1,z_2)&=
&\psi^{k_1,k_2}_{l_1,l_2,l_3}(z_1,z_2)\qquad
((l_1,l_2,l_3)\in R_V^{k_1,k_2}),
\notag
\ena
and therefore surjections in Lemmas \ref{surj1} and \ref{surj2} are
isomorphisms of $\hat{\mathfrak{n}}$-modules.

Moreover, Theorem \ref{thm:chrec} holds.
\end{thm}
\begin{proof}
The theorem immediately follows from Corollary \ref{cor} applied to the series
\be
\bar{\varphi}^{k_1,k_2}_{l_1,l_2,l_3}(z_1,z_2)=
\varphi^{k_1,k_2}_{l_1,l_2,l_3}(z_1,z_2)-
(\varphi_{VO})^{k_1,k_2}_{l_1,l_2,l_3}(z_1,z_2),\\
\bar{\psi}^{k_1,k_2}_{l_1,l_2,l_3}(z_1,z_2)=
\psi^{k_1,k_2}_{l_1,l_2,l_3}(z_1,z_2)-
(\psi_{VO})^{k_1,k_2}_{l_1,l_2,l_3}(z_1,z_2).
\en
\end{proof}

\begin{cor}\label{V=Vpr}
The principal subspace $V^{k}_{l_1,l_2}\subset M^k_{l_1,l_2}$ is isomorphic to
$V^{k,k}_{l_1,l_2,l_1+l_2}$.
\end{cor}
\begin{proof}
As we noted in Remark \ref{remV=V}, there is a
surjective homomorphism of $\nh$-modules
\be
V^{k,k}_{l_1,l_2,l_1+l_2}\to V^{k}_{l_1,l_2}\to 0.
\en
On the other hand, taking tensor products of the Frenkel-Kac construction
we obtain a surjection
\be
V^{k}_{l_1,l_2}\to (V_{VO})^{k,k}_{l_1,l_2,l_1+l_2}\to 0.
\en
Hence the assertion follows from Theorem \ref{thm:equalities}.
\end{proof}

\subsection{Other cases}
In this section we describe two more versions of Corollary \ref{cor}
which we use later to establish the bosonic formulas for our
characters.

Set
\bea
{\widetilde R}_U^{k_1,k_2}=\{(l_1,l_2,l_3) \ |\
k_1-1\leq l_1+l_2-l_3\leq k_2\}\cap
P_U^{k_1,k_2}.
\label{Rtilde}
\ena

\begin{prop}\label{unique2}
Let $\bar{\varphi}^{k_1,k_2}_{l_1,l_2,l_3}(z_1,z_2)$
($(l_1,l_2,l_3)\in {\widetilde
R}_U^{k_1,k_2}$) and $\bar{\psi}^{k_1,k_2}_{l_1,l_2,l_3}(z_1,z_2)$
($(l_1,l_2,l_3)\in {R}_V^{k_1,k_2}$) be formal power series in variables
$z_1,z_2$
such that equations (\ref{TRa})-(\ref{TRd})
are satisfied (where (\ref{TRc}) is assumed for
$(l_1,l_2,l_3)\in {\widetilde R}_U^{k_1,k_2}$ and either
$l_1+l_2-l_3\neq k_2$ or $l_3=0$).

Assume that all coefficients of the power series are Laurent
power series in $q$.
Assume also
that the formal power series $\bar{\psi}^{k_1,k_2}_{0,0,0}(z_1,z_2)$ has no
constant term.

Then all these formal power series are identically zero.
\end{prop}
\begin{proof}
The proof of the proposition is similar to the proof of Corollary \ref{cor}.
\end{proof}

\medskip

Consider the case $k_1=k_2=k$.
Then if $(l_1,l_2,l_3)\in {R}_U^{k,k}$, then $l_3=l_1+l_2-k$.
And if $(l_1,l_2,l_3)\in {R}_V^{k,k}$, then $l_3=l_1+l_2$.

Let $\bar{\varphi}^{k,k}_{l_1,l_2,l_3}(z_1,z_2)$
($(l_1,l_2,l_3)\in {R}_U^{k,k}$) and $\bar{\psi}^{k,k}_{l_1,l_2,l_3}(z_1,z_2)$
($(l_1,l_2,l_3)\in {R}_V^{k,k}$) be formal power series in variables
$z_1,z_2$
such that the equations
\bea
\bar{\psi}_{l_1,l_2,l_1+l_2}^{k,k}(z_1,z_2)=
\bar{\psi}_{l_1,l_2-1,l_1+l_2-1}^{k,k}(z_1,z_2)+
z_2^{l_2}
\bar{\varphi}_{l_1+l_2,k-l_2,l_1}^{k,k}(z_1,qz_2),\notag\\
\bar{\psi}_{l_1,l_2,l_1+l_2}^{k,k}(z_1,z_2)=
\bar{\psi}_{l_1-1,l_2,l_1+l_2-1}^{k,k}(z_1,z_2)+
z_1^{l_1}
\bar{\varphi}_{k-l_1,l_1+l_2,l_2}^{k,k}(qz_1,z_2)
\notag
\ena
and
\be
\bar{\varphi}_{l_1,l_2,l_1+l_2-k}^{k,k}(z_1,z_2)=
\bar{\varphi}_{l_1,l_2-1,l_1+l_2-k-1}^{k,k}(z_1,z_2)
{\hspace{100pt}}\\
+z_2^{l_2}
\bar{\psi}_{l_1+l_2-k,k-l_2,l_1}^{k,k}(q^{-1}z_1,qz_2)
{\hspace{90pt}}\\
{\hspace{50pt}}+(q^{-1}z_1z_2)^{l_1+l_2-k}
\bar{\psi}_{k-l_2,k-l_1-1,2k-l_1-l_2-1}^{k,k}(z_1,z_2)
\en
are satisfied.

Note that the last equation is obtained from equations
(\ref{TRa}) and (\ref{TRc}) via eliminating the term
$\bar \varphi_{l_1,l_2,l_1+l_2-k-1}^{k,k}(z_1,z_2)$.

\begin{prop}\label{unique3}
Assume that all coefficients of all power series
$\bar{\varphi}^{k,k}_{l_1,l_2,l_3}(z_1,z_2)$ ($(l_1,l_2,l_3)\in
{R}_U^{k,k}$) and $\bar{\psi}^{k,k}_{l_1,l_2,l_3}(z_1,z_2)$
($(l_1,l_2,l_3)\in {R}_V^{k,k}$)  are Laurent
power series in $q$.
Assume also
that the formal power series
$\bar{\psi}^{k,k}_{0,0,0}(z_1,z_2)$ has no constant term.

Then all these formal power series are identically zero.
\end{prop}
\begin{proof}
The proof of the proposition is similar to the proof of Corollary \ref{cor}.
\end{proof}

\section{Bosonic formulas for the characters of $\nh$-modules}
In this section we write explicit solutions of recursion
relations \eqref{TRa}--\eqref{TRd} in the regions $\tilde
R^{k_1,k_2}_U$ and $R^{k_1,k_2}_V$ in the bosonic form. First, we
prepare notation and recall basic facts about the small principal
$\widehat{\mathfrak{sl}_3}$ subspaces (see \cite{FJLMM1}).

\subsection{The small principal subspaces.}
Let $\ah$ denote the abelian Lie algebra
spanned by $e_{21}[n]$, $e_{31}[n]$,
$n\in\Z$.
{}For non-negative integers $k,l_1,l_2$ satisfying
$l_1+l_2\le k$, define $X_{l_1,l_2}^k$ to
be the cyclic $\ah$-module with a cyclic vector $v$ and
the defining relations
\be
&&e_{21}[n]v=e_{31}[n]v=0
\quad (n>0),
\\
&&e_{21}[0]^{l_1+1}v=0,
\\
&&e_{21}[0]^\alpha e_{31}[0]^\beta v=0
\quad (\alpha+\beta= l_1+l_2+1),
\\
&&e_{21}(z)^\alpha e_{32}(z)^\beta=0
\quad (\alpha+\beta=k+1).
\en

The space $X_{l_1,l_2}^k$ has a monomial basis of the form
$$
\dots e_{31}[-1]^{a_3}e_{21}[-1]^{a_2}
e_{31}[0]^{a_1}e_{21}[0]^{a_0}v,
$$
where $\{a_i\}_{i \ge 0}$ run over
sequences of non-negative integers such that
$a_i=0$ for almost all $i$ and that satisfy the conditions
\begin{gather*}
a_0\le l_1,\ a_0+a_1\le l_1+l_2,\\
a_i+a_{i+1}+a_{i+2}\le k.
\end{gather*}

Let $\chi_{l_1,l_2}^k(z_1,z_2)$ denote the character of $X_{l_1,l_2}^k$
(normalized in such a way that the degree of the cyclic vector $v$ is
$(0,0,0)$).
The description of the monomial basis of $X_{l_1,l_2}^k$  leads to the
following recursion relations
\begin{equation}\label{sr}
\chi_{l_1,l_2}^k(z_1,z_2)=
\chi_{l_1-1,l_2}^k(z_1,z_2)
+z_1^{l_1} \chi^k_{l_2,k-l_1-l_2}(z_1,qz_2).
\end{equation}

We now write a formula for $\chi_{l_1,l_2}^k(z_1,z_2)$.
Let the quantities $p(m,n,s,z_1,z_2)$ $(m,n\in\Z_{\geq 0},\ 0\le s\le 5)$ be
given by
\bea
&&p^k_{l_1,l_2}(m,n,0,z_1,z_2)=z_1^{km}z_2^{kn}, \nonumber
\\
&&p^k_{l_1,l_2}(m,n,1,z_1,z_2)=z_1^{km+l_1}z_2^{kn},
\nonumber\\
&&p^k_{l_1,l_2}(m,n,2,z_1,z_2)=z_1^{km+l_1+l_2}z_2^{kn+l_2},
\nonumber\\
&&p^k_{l_1,l_2}(m,n,3,z_1,z_2)=z_1^{km+l_1+l_2}z_2^{kn+l_1+l_2},
\nonumber\\
&&p^k_{l_1,l_2}(m,n,4,z_1,z_2)=z_1^{km+l_1}z_2^{kn+l_1+l_2},
\nonumber\\
&&p^k_{l_1,l_2}(m,n,5,z_1,z_2)=z_1^{km}z_2^{kn+l_2}.
\nonumber
\ena
Let the quantities $a(m,n,s)$ $(m,n\in\Z_{\geq 0},\ 0\le s\le 5)$ be
given by
\bea
&&a^k_{l_1,l_2}(m,n,0)=kQ_2(m,n)-ml_1-nl_2, \nonumber
\\
&&a^k_{l_1,l_2}(m,n,1)=kQ_2(m,n)+(m-n)l_1-nl_2,
\nonumber\\
&&a^k_{l_1,l_2}(m,n,2)=kQ_2(m,n)+(m-n)l_1+ml_2,
\nonumber\\
&&a^k_{l_1,l_2}(m,n,3)=kQ_2(m,n)+nl_1+ml_2,
\nonumber\\
&&a^k_{l_1,l_2}(m,n,4)=kQ_2(m,n)+nl_1+(-m+n)l_2,
\nonumber\\
&&a^k_{l_1,l_2}(m,n,5)=kQ_2(m,n)-ml_1+(-m+n)l_2,
\nonumber
\ena
where $Q_2(m,n)=m^2+n^2-mn.$

Let the quantities $d(m,n,s,z_1,z_2)$ $(m,n\in\Z_{\geq 0},\ 0\le s\le 5)$ be
given by
\be
d(m,n,0,z_1,z_2)&=&(q)_n(q)_{m-n}(z_1q^{2m-n})_\infty(z_1^{-1}q^{-2m+n+1})_{m-n}\\
&&\times(z_1z_2q^{m+n})_\infty(z_1^{-1}z_2^{-1}q^{-m-n+1})_n
(z_2q^{2n-m})_{m-n}(z_2^{-1}q^{-2n+m+1})_n,\\
d(m,n,1,z_1,z_2)&=&(q)_n(q)_{m-n}(z_1q^{2m-n+1})_\infty(z_1^{-1}q^{-2m+n})_{m-n+1}\\
&&\times(z_1z_2q^{m+n})_\infty(z_1^{-1}z_2^{-1}q^{-m-n+1})_n
(z_2q^{2n-m})_{m-n}(z_2^{-1}q^{-2n+m+1})_n,\\
d(m,n,2,z_1,z_2)&=&(q)_n(q)_{m-n}(z_1q^{2m-n+1})_\infty
(z_1^{-1}q^{-2m+n})_{m-n}\\
&&\times(z_1z_2q^{m+n+1})_\infty(z_1^{-1}z_2^{-1}q^{-m-n})_{n+1}
(z_2q^{2n-m})_{m-n+1}(z_2^{-1}q^{-2n+m+1})_n,\\
d(m,n,3,z_1,z_2)&=&(q)_n(q)_{m-n}(z_1q^{2m-n+1})_\infty
(z_1^{-1}q^{-2m+n})_{m-n}\\
&&\times(z_1z_2q^{m+n+1})_\infty(z_1^{-1}z_2^{-1}q^{-m-n})_{n+1}
(z_2q^{2n-m+1})_{m-n}(z_2^{-1}q^{-2n+m})_{n+1},\\
d(m,n,4,z_1,z_2)&=&(q)_n
(q)_{m-n-1}(z_1q^{2m-n})_\infty(z_1^{-1}q^{-2m+n+1})_{m-n}\\
&&\times(z_1z_2q^{m+n+1})_\infty(z_1^{-1}z_2^{-1}q^{-m-n})_{n+1}
(z_2q^{2n-m+1})_{m-n}(z_2^{-1}q^{-2n+m})_{n+1},\\
d(m,n,5,z_1,z_2)&=&(q)_n
(q)_{m-n-1}(z_1q^{2m-n})_\infty(z_1^{-1}q^{-2m+n+1})_{m-n}\\
&&\times(z_1z_2q^{m+n})_\infty(z_1^{-1}z_2^{-1}q^{-m-n+1})_n
(z_2q^{2n-m+1})_{m-n}(z_2^{-1}q^{-2n+m})_{n+1}.
\en

{}For all integers $k,l_1,l_2$, we define
\begin{equation}\label{chiB}
(\chi_B)^k_{l_1,l_1+l_2}(z_1,z_2)=
\sum_{m\geq n\geq0,s=0,\ldots,5}\frac{p^k_{l_1,l_2}(m,n,s,z_1,z_2)
q^{a^k_{l_1,l_2}(m,n,s)}}{d(m,n,s,z_1,z_2)}.
\end{equation}

\begin{rem}
We deal with  expressions of the form
$(\chi_B)^k_{l_1,l_1+l_2}(q^\alpha z_1,q^\beta z_2)$,
$(\chi_B)^k_{l_1,l_1+l_2}(q^\alpha z_1z_2, q^\beta z_2^{-1})$ etc.
All these expressions are sums where each term is a ratio of a monomial in $z_1,z_2$ and of a product of factors of the form $(1-z_1^iz_2^jq^k)$,
where either $i\ge 0, j\ge 0$ or $i< 0$, $j< 0$. In the first case we expand
$$\frac{1}{1-z_1^iz_2^jq^k}=\sum_{\al=0}^\infty (z_1^iz_2^jq^k)^\alpha$$
and in the second case
$$\frac{1}{1-z_1^iz_2^jq^k}=(-z_1^{-i}z_2^{-j}q^{-k})\sum_{\al=0}^\infty (z_1^iz_2^jq^k)^{-\alpha}.$$
Using these expansions we can always rewrite our expressions
as formal power series in the variables $z_1,z_2$ whose
coefficients are Laurent power series in $q$.
\end{rem}

The following result can be extracted from
\cite{FJLMM2}. For completeness we give a proof.
\begin{thm}\label{chsp}
{}For non-negative integers $k,l_1,l_2$ such that $l_1+l_2\leq k$, we have
\be
\chi^k_{l_1,l_2}(z_1,z_2)=(\chi_B)^k_{l_1,l_2}(z_1,z_2).
\en
\end{thm}

\begin{proof}
The set of
formal power series $\chi_{l_1,l_2}^k(z_1,z_2)$
($0\le l_1,l_2,l_1+l_2\le k$)
are uniquely determined
by
\begin{enumerate}
\item relations (\ref{sr}),
\item the normalization $\chi_{l_1,l_2}^k(0,0)=1$,
\item the initial condition $\chi_{-1,l_2}^k(z_1,z_2)=0$.
\end{enumerate}
All these conditions can be verified for
$(\chi_B)^k_{l_1,l_2}(z_1,z_2)$ by a direct computation.
\end{proof}

\subsection{The bosonic formula for the $\A$-modules.}
{}For all integers $k_1,k_2,l_1,l_2,l_3$
we introduce the formal power series in $z_1$, $z_2$ whose
coefficients are Laurent power series in $q$:
\bea
(\varphi_B)^{k_1,k_2}_{l_1,l_2,l_3}(z_1,z_2)
&=&\sum_{i\geq0}
\frac{z_2^{ik_2}q^{i^2k_2-il_2}}
{(q)_i(q^{2i}z_2)_\infty(q^{-2i+1}z_2^{-1})_i}
(\chi_B)^{k_1}_{l_3,l_1}(q^{i-1}z_1z_2,q^{-2i+1}z_2^{-1})
\label{PHIB}\\
&&+\sum_{i\geq0}
\frac{z_2^{ik_2+l_2}q^{i^2k_2+il_2}}
{(q)_i(q^{2i+1}z_2)_\infty(q^{-2i}z_2^{-1})_{i+1}}
(\chi_B)^{k_1}_{l_3,l_1}(q^{-i-1}z_1,q^{2i+1}z_2),
\nonumber\\
(\psi_B)^{k_1,k_2}_{l_1,l_2,l_3}(z_1,z_2)
&=&\sum_{i\geq0}
\frac{z_2^{ik_2}q^{i^2k_2-il_2}}
{(q)_i(q^{2i}z_2)_\infty(q^{-2i+1}z_2^{-1})_i}
(\chi_B)^{k_1}_{l_1,l_3}(q^{-i}z_1,q^{2i}z_2)
\label{PSIB}\\
&&+\sum_{i\geq0}
\frac{z_2^{ik_2+l_2}q^{i^2k_2+il_2}}
{(q)_i(q^{2i+1}z_2)_\infty(q^{-2i}z_2^{-1})_{i+1}}
(\chi_B)^{k_1}_{l_1,l_3}(q^iz_1z_2,q^{-2i}z_2^{-1}).
\nonumber
\ena
\begin{prop}
{}For all integers  $k_1,k_2,l_1,l_2,l_3$ we have
\bea
&&(\varphi_B)^{k_1,k_2}_{l_1,l_2,l_3}(z_1,z_2)
=(\varphi_B)^{k_1,k_2}_{l_1,l_2-1,l_3}(z_1,z_2)+
z_2^{l_2}(\psi_B)^{k_1,k_2}_{l_3,k_2-l_2,l_1}(q^{-1}z_1,qz_2),
\label{Ba}\\
&&(\psi_B)^{k_1,k_2}_{l_1,l_2,l_3}(z_1,z_2)
=(\psi_B)^{k_1,k_2}_{l_1,l_2-1,l_3}(z_1,z_2)+
z_2^{l_2}(\varphi_B)^{k_1,k_2}_{l_3,k_2-l_2,l_1}(z_1,qz_2),
\label{Bb}\\
&&(\varphi_B)^{k_1,k_2}_{l_1,l_2,l_3}(z_1,z_2)
=(\varphi_B)^{k_1,k_2}_{l_1,l_2,l_3-1}(z_1,z_2)
+(q^{-1}z_1z_2)^{l_3}
(\psi_B)^{k_1,k_2}_{l_1-l_3,l_2-l_3,k_1-l_3}(z_1,z_2),
\label{Bc}\\
&&(\psi_B)^{k_1,k_2}_{l_1,l_2,l_3}(z_1,z_2)
=(\psi_B)^{k_1,k_2}_{l_1-1,l_2,l_3}(z_1,z_2)+
z_1^{l_1}(\varphi_B)^{k_1,k_2}_{k_1-l_1,l_1+l_2,l_3-l_1}
(qz_1,z_2).
\label{Bd}
\ena
\end{prop}
\begin{proof}
We prove \eqref{Ba}.
The proof of other formulas is similar.

We have
\be
&&(\varphi_B)^{k_1,k_2}_{l_1,l_2,l_3}(z_1,z_2)-
z_2^{l_2}(\psi_B)^{k_1,k_2}_{l_3,k_2-l_2,l_1}(q^{-1}z_1,qz_2)\\
&&\quad=\sum_{i\geq0}
\frac{z_2^{ik_2}q^{i^2k_2-il_2}}{(q)_i(q^{2i}z_2)_\infty(q^{-2i+1}z_2^{-1})_i}
(\chi_B)^{k_1}_{l_3,l_1}(q^{i-1}z_1z_2,q^{-2i+1}z_2^{-1})[1-(1-q^i)]\\
&&\qquad+\sum_{i\geq0}
\frac{z_2^{ik_2+l_2}q^{i^2k_2+il_2}}
{(q)_i(q^{2i+1}z_2)_\infty(q^{-2i}z_2^{-1})_{i+1}}
(\chi_B)^{k_1}_{l_3,l_1}(q^{-i-1}z_1,q^{2i+1}z_2)[1-(1-q^{-i}z_2^{-1})]\\
&&\quad=(\varphi_B)^{k_1,k_2}_{l_1,l_2-1,l_3}(z_1,z_2).
\en
\end{proof}

\begin{thm}\label{boson}
If $(l_1,l_2,l_3)\in \tilde R_U^{k_1,k_2}$, then
$$
(\varphi_B)_{l_1,l_2,l_3}^{k_1,k_2}(z_1,z_2)=\varphi_{l_1,l_2,l_3}^{k_1,k_2}(z_1,z_2).
$$
If $(l_1,l_2,l_3)\in R_V^{k_1,k_2}$, then
$$
(\psi_B)_{l_1,l_2,l_3}^{k_1,k_2}(z_1,z_2)=\psi_{l_1,l_2,l_3}^{k_1,k_2}(z_1,z_2).
$$
\end{thm}
\begin{proof}
Consider $\varphi_{l_1,l_2,l_3}^{k_1,k_2}(z_1,z_2)$
$((l_1,l_2,l_3)\in \tilde R^{k_1,k_2}_U)$ and
$\psi_{l_1,l_2,l_3}^{k_1,k_2}(z_1,z_2)$
$((l_1,l_2,l_3)\in R^{k_1,k_2}_V)$.
They satisfy recursion relations \eqref{TRa}--\eqref{TRd}.\\
The series
$(\varphi_B)_{l_1,l_2,l_3}^{k_1,k_2}(z_1,z_2)$
$((l_1,l_2,l_3)\in \tilde R^{k_1,k_2}_U)$ and
$(\psi_B)_{l_1,l_2,l_3}^{k_1,k_2}(z_1,z_2)$
$((l_1,l_2,l_3)\in R^{k_1,k_2}_V)$ satisfy relations
\eqref{Ba}--\eqref{Bd}.

Because of the uniqueness of the solution
(see Proposition \ref{unique2}) it is enough to check that
\begin{enumerate}
\item[1)] \label{1} $(\psi_B)_{-1,l_2,l_3}^{k_1,k_2}(z_1,z_2)=0,$
\item[2)] \label{2} $(\psi_B)_{l_1,-1,l_3}^{k_1,k_2}(z_1,z_2)=0,$
\item[3)] \label{3} $(\varphi_B)_{l_1,-1,l_3}^{k_1,k_2}(z_1,z_2)=0,$
\item[4)] \label{4} $(\varphi_B)_{l_1,l_2,-1}^{k_1,k_2}(z_1,z_2)=0,$
\item[5)] \label{6}
 $(\varphi_B)_{k_1,l_2,\min(k_1,l_2)}^{k_1,k_2}(z_1,z_2)=(\varphi_B)_{k_1,l_2,\min(k_1,l_2)+1}^{k_1,k_2}(z_1,z_2),$
\item[6)] \label{5}
$(\psi_B)_{l_1,l_2,l_1+l_2}^{k_1,k_2}(z_1,z_2)=(\psi_B)_{l_1,l_2,l_1+l_2+1}^{k_1,k_2}(z_1,z_2).$
\end{enumerate}
We prove the last formula. The proof of the rest is similar.
We need to show that
\bea
&&\sum_{i\geq0}
\frac{z_2^{ik_2}q^{i^2k_2-il_2}}{(q)_i(q^{2i}z_2)_\infty(q^{-2i+1}z_2^{-1})_i}
(\chi_B)^{k_1}_{l_1,l_1+l_2}(q^{-i}z_1,q^{2i}z_2)
\nonumber
\\
&&\quad
+\sum_{i\geq0}
\frac{z_2^{ik_2+l_2}q^{i^2k_2+il_2}}
{(q)_i(q^{2i+1}z_2)_\infty(q^{-2i}z_2^{-1})_{i+1}}
(\chi_B)^{k_1}_{l_1,l_1+l_2}(q^iz_1z_2,q^{-2i}z_2^{-1})
\nonumber\\
&&
=\sum_{i\geq0}
\frac{z_2^{ik_2}q^{i^2k_2-il_2}}
{(q)_i(q^{2i}z_2)_\infty(q^{-2i+1}z_2^{-1})_i}
(\chi_B)^{k_1}_{l_1,l_1+l_2+1}(q^{-i}z_1,q^{2i}z_2)
\nonumber\\
&&\quad
+\sum_{i\geq0}
\frac{z_2^{ik_2+l_2}q^{i^2k_2+il_2}}
{(q)_i(q^{2i+1}z_2)_\infty(q^{-2i}z_2^{-1})_{i+1}}
(\chi_B)^{k_1}_{l_1,l_1+l_2+1}(q^iz_1z_2,q^{-2i}z_2^{-1}).
\nonumber
\ena
It is sufficient to show that the
coefficients of $z_2^{ik_2}q^{i^2k_2}$
for each $i$ are equal:
\begin{multline}\label{4chi}
-q^{-i(l_2+2)}z_2^{-1}
(\chi_B)^{k_1}_{l_1,l_1+l_2}(q^{-i}z_1,q^{2i}z_2)+
q^{il_2} z_2^{l_2}
(\chi_B)^{k_1}_{l_1,l_1+l_2}(q^iz_1z_2,q^{-2i}z_2^{-1})\\=
-q^{-i(l_2+2)}z_2^{-1}
(\chi_B)^{k_1}_{l_1,l_1+l_2+1}(q^{-i}z_1,q^{2i}z_2)+
q^{il_2} z_2^{l_2}
(\chi_B)^{k_1}_{l_1,l_1+l_2+1}(q^iz_1z_2,q^{-2i}z_2^{-1}).
\end{multline}
Each term in \eqref{4chi} has the form
$\sum_{m\ge n\ge 0}\sum_{s=0}^5 p^k_{l_1,l_2}(m,n,s,z_1,z_2)
g_s(z_1,z_2,q)$ where $g_s$ are
independent of $k_1$, $l_1$ and $l_2$.

Equating coefficients of
$p^k_{l_1,l_2}(m,n,0,z_1,z_2)$ we are led to show
\begin{multline}\label{tildea}
\frac{-q^{-2i}
q^{a^{k_1}_{l_1,l_2}(m,n,0)}}{d(m,n,0,q^{-i}z_1,q^{2i}z_2)}+
\frac{z_2
q^{a^{k_1}_{l_1,l_2}(m,m-n,5)}}{d(m,m-n,5,q^iz_1z_2,q^{-2i}z_2^{-1})}\\=
\frac{-q^{-2i}
q^{a^{k_1}_{l_1,l_2+1}(m,n,0)}}{d(m,n,0,q^{-i}z_1,q^{2i}z_2)}+
\frac{z_2
q^{a^{k_1}_{l_1,l_2+1}(m,m-n,5)}}{d(m,m-n,5,q^iz_1z_2,q^{-2i}z_2^{-1})}.
\end{multline}
Equation \eqref{tildea} is equivalent to the equation
$$
\frac{d(m,n,0,q^{-i}z_1,q^{2i}z_2)}{d(m,m-n,5,q^iz_1z_2,q^{-2i}z_2^{-1})}=\frac{1-q^n}{1-z_2q^{n+2i}},
$$
which can be checked by a direct calculation.
In a similar way we check that the coefficients of
all monomials $p^k_{l_1,l_2}(m,n,s,z_1,z_2)$
coincide.
The theorem is proved.
\end{proof}

\section{Case of $k_1=k_2$ and Toda recursion}
In this section we restrict to the case of $k_1=k_2=k$.  If
$(l_1,l_2,l_3)\in R^{k,k}_V$, then $l_3=l_1+l_2$, and if
$(l_1,l_2,l_3)\in R^{k,k}_U$, then $l_3=l_1+l_2-k$.  As a result, in
\eqref{PHIB} and \eqref{PSIB} several terms have the same dependence
on $k,l_1,l_2$.  In principle, these terms can be summed up. However,
a direct summation is not completely obvious. We use our recursion to
obtain the result of the summation, which turns out to have a
factorized form.

Set
\bea
I_{d_1,d_2}(z_1,z_2)=\frac{(q z_1^{-1} z_2^{-1})_{d_1+d_2}}
{(q)_{d_1} (q)_{d_2}(qz_1^{-1})_{d_1} (qz_2^{-1})_{d_2}
(q z_1^{-1} z_2^{-1})_{d_1}(q z_1^{-1} z_2^{-1})_{d_2}}.
\label{Idd}
\ena

\begin{prop}\label{prop:Toda}
The functions $I_{d_1,d_2}(z_1,z_2)$ satisfy
the following recurrence relation
\bea
&&(z_1^{-1}(q^{d_1}-1)+(q^{d_2-d_1}-1)+
z_2(q^{-d_2}-1))I_{d_1,d_2}(z_1,z_2)\nonumber 
\\
&&\quad=
q^{d_2-d_1}I_{d_1-1,d_2}(z_1,z_2)+
                 z_2q^{-d_2}I_{d_1,d_2-1}(z_1,z_2).\nonumber
\ena
\end{prop}
We call this relation the {\it Toda recursion}.

Set further
\be
\bar{J}_{d_1,d_2}(z_1,z_2)
=\frac1{(z_1)_\infty(z_2)_\infty(z_1z_2)_\infty}
I_{d_1,d_2}(z_1,z_2).
\en

\begin{definition}\label{AB}
We define the series
$A^s_{d_1,d_2}(z_1,z_2)$, $B^s_{d_1,d_2}(z_1,z_2)$ as follows:
\be
&&A^s_{d_1,d_2}(z_1,z_2)
=\bar{J}_{d_1,d_2}(w_1,w_2)f^s(w_1,w_2),
\\
&&B^0_{d_1,d_2}(z_1,qz_2)
=(1-z_2^{-1}q^{d_1-d_2})(-w_2)\bar{J}_{d_1,d_2}(w_1,w_2),
\\
&&B^1_{d_1,d_2}(z_1,qz_2)=
(1-z_1^{-1}z_2^{-1}q^{-d_1})w_1^2w_2
\bar{J}_{d_1,d_2}(w_1,w_2),
\\
&&B^2_{d_1-1,d_2-1}(z_1,qz_2)=(1-q^{d_2})(-w_1)
\bar{J}_{d_1,d_2}(w_1,w_2),
\\
&&B^3_{d_1-1,d_2-1}(z_1,qz_2)=(1-q^{d_1-d_2}z_2^{-1})
               w_1w_2^2 \bar{J}_{d_1,d_2}(w_1,w_2),
\\
&&B^4_{d_1,d_2-1}(z_1,qz_2)=(1-q^{-d_1}z_1^{-1}z_2^{-1})
               (-w_1^2w_2^2) \bar{J}_{d_1,d_2}(w_1,w_2),
\\
&&B^5_{d_1,d_2-1}(z_1,qz_2)=(1-q^{d_2}) \bar{J}_{d_1,d_2}(w_1,w_2),
\en
where $w_1=z_1q^{2d_1-d_2}$, $w_2=z_2q^{2d_2-d_1}$, and
\be
&&f^0(w_1,w_2)=1, f^1(w_1,w_2)=-w_1,
f^2(w_1,w_2)=w_1^2w_2,\\
&&
f^3(w_1,w_2)=-w_1^2w_2^2,
f^4(w_1,w_2)=w_1w_2^2,
f^5(w_1,w_2)=-w_2.
\en
\end{definition}

In this section we prove the following theorem.
\begin{thm}
We have
\label{g-l}
\bea
\label{psi}\\
&&\psi_{l_1,l_2,l_1+l_2}^{k,k}(z_1,z_2)=
\sum_{d_1,d_2\ge 0} z_1^{kd_1} z_2^{kd_2} q^{k(d_1^2+d_2^2-d_1d_2)}\nonumber
\\&&\times\Bigl(
q^{-l_1d_1-l_2d_2} A^0_{d_1,d_2}(z_1,z_2)
+z_1^{l_1} q^{l_1(d_1-d_2)-l_2d_2}A^1_{d_1,d_2}(z_1,z_2)
+z_1^{l_1+l_2}z_2^{l_2} q^{l_1(d_1-d_2)+l_2d_1} A^2_{d_1,d_2}(z_1,z_2)\nonumber\\
&&\quad+(z_1z_2)^{l_1+l_2}q^{l_1d_2+l_2d_1} A^3_{d_1,d_2}(z_1,z_2)
+z_1^{l_1}z_2^{l_1+l_2}q^{l_2(d_2-d_1)+l_1d_2} A^4_{d_1,d_2}(z_1,z_2)\nonumber\\
&&\quad+z_2^{l_2} q^{-l_1d_1-l_2(d_1-d_2)} A^5_{d_1,d_2}(z_1,z_2)\Bigr)\nonumber
\ena
and
\bea
\label{phi}\\
&&\varphi_{l_1,l_2,l_1+l_2-k}^{k,k}(z_1,z_2)= \sum_{d_1,d_2\ge 0}
z_1^{kd_1} z_2^{kd_2} q^{k(d_1^2+d_2^2-d_1d_2)}\nonumber
\\&&\times \Bigl(
q^{-l_1d_1-l_2d_2} B^0_{d_1,d_2}(z_1,z_2)
+z_1^{l_1} q^{l_1(d_1-d_2)-l_2d_2} B^1_{d_1,d_2}(z_1,z_2)
+z_1^{l_1+l_2}z_2^{l_2} q^{l_1(d_1-d_2)+l_2d_1} B^2_{d_1,d_2}(z_1,z_2)\nonumber\\
&&\quad+(z_1z_2)^{l_1+l_2}q^{l_1d_2+l_2d_1} B^3_{d_1,d_2}(z_1,z_2)
+z_1^{l_1}z_2^{l_1+l_2}q^{l_2(d_2-d_1)+l_1d_2} B^4_{d_1,d_2}(z_1,z_2)\nonumber\\
&&\quad+z_2^{l_2} q^{-l_1d_1-l_2(d_1-d_2)} B^5_{d_1,d_2}(z_1,z_2)\Bigr).\nonumber
\ena
\end{thm}

The strategy of the proof is as follows.
We show that the right hand sides of
\eqref{psi} and \eqref{phi} satisfy
the recursion relations provided 
certain relations for $A^s_{d_1,d_2}$ and
$B^s_{d_1,d_2}$ are satisfied.
We prove that these relations hold  for
$A^s_{d_1,d_2}$, $B^s_{d_1,d_2}$
of the form given in Definition \ref{AB}.
This proves Theorem  \ref{g-l} because of the uniqueness of the
solution of the recursion relations (Proposition \ref{unique3}).

Substituting \eqref{psi} and \eqref{phi} into the recursion relations, 
one can easily verify the following three Lemmas. 

\begin{lem}
\label{lem:1}
The relations
\be
&&(1-q^{d_2})A^0_{d_1,d_2}(z_1,z_2)=B^5_{d_1,d_2-1}(z_1,qz_2),
\\
&&(1-q^{d_2})A^1_{d_1,d_2}(z_1,z_2)=B^2_{d_1-1,d_2-1}(z_1,qz_2),
\\
&&(1-q^{-d_1}z_1^{-1}z_2^{-1})A^2_{d_1,d_2}(z_1,z_2)
=B^1_{d_1,d_2}(z_1,qz_2),
\\
&&(1-q^{-d_1}z_1^{-1}z_2^{-1})A^3_{d_1,d_2}(z_1,z_2)
=B^4_{d_1,d_2-1}(z_1,qz_2),
\\
&&(1-q^{d_1-d_2}z_2^{-1})A^4_{d_1,d_2}(z_1,z_2)
=B^3_{d_1-1,d_2-1}(z_1,qz_2),
\\
&&(1-q^{d_1-d_2}z_2^{-1})A^5_{d_1,d_2}(z_1,z_2)
=B^0_{d_1,d_2}(z_1,qz_2)
\en
imply the recursion
\begin{equation}\label{bar1}
\psi_{l_1,l_2,l_1+l_2}^{k,k}(z_1,z_2)=
\psi_{l_1,l_2-1,l_1+l_2-1}^{k,k}(z_1,z_2)+
z_2^{l_2} \varphi_{l_1+l_2,k-l_2,l_1}^{k,k}(z_1,qz_2).
\end{equation}
\end{lem}

\begin{lem}
\label{lem:2}
The relations
\be
&&(1-q^{d_1})A^0_{d_1,d_2}(z_1,z_2)=B^1_{d_1-1,d_2}(qz_1,z_2),
\\
&&(1-q^{d_2-d_1}z_1^{-1})A^1_{d_1,d_2}(z_1,z_2)
=B^0_{d_1,d_2}(qz_1,z_2),
\\
&&(1-q^{d_2-d_1}z_1^{-1})A^2_{d_1,d_2}(z_1,z_2)
=B^3_{d_1-1,d_2-1}(qz_1,z_2),
\\
&&(1-q^{-d_2}z_1^{-1}z_2^{-1})A^3_{d_1,d_2}(z_1,z_2)
=B^2_{d_1-1,d_2}(qz_1,z_2),
\\
&&(1-q^{-d_2}z_1^{-1}z_2^{-1})A^4_{d_1,d_2}(z_1,z_2)
=B^5_{d_1,d_2}(qz_1,z_2),
\\
&&(1-q^{d_1})A^5_{d_1,d_2}(z_1,z_2)
=B^4_{d_1-1,d_2-1}(qz_1,z_2)
\en
imply the recursion
\begin{equation}\label{bar2}
\psi_{l_1,l_2,l_1+l_2}^{k,k}(z_1,z_2)=
\psi_{l_1-1,l_2,l_1+l_2-1}^{k,k}(z_1,z_2)+
z_1^{l_1} \varphi_{k-l_1,l_1+l_2,l_2}^{k,k}(qz_1,z_2).
\end{equation}
\end{lem}

\begin{lem}
\label{lem:3}
The relations
\be
&&(1-q^{d_2})B^0_{d_1,d_2}(z_1,z_2)
=z_1^{-1}z_2^{-1}q^{-d_1+1}A^3_{d_1-1,d_2-1}(z_1,z_2)+
A^5_{d_1,d_2-1}(q^{-1}z_1,qz_2),
\\
&&(1-q^{d_2})B^1_{d_1,d_2}(z_1,z_2)=
z_2^{-1}q^{d_1-d_2+1}A^4_{d_1,d_2-1}(z_1,z_2)+
A^2_{d_1,d_2-1}(q^{-1}z_1,qz_2),
\\
&&(1-q^{-d_1}z_1^{-1}z_2^{-1})B^2_{d_1,d_2}(z_1,z_2)=
z_2^{-1}q^{d_1-d_2+1}A^5_{d_1+1,d_2}(z_1,z_2)+
A^1_{d_1+1,d_2}(q^{-1}z_1,qz_2),
\\
&&(1-q^{-d_1}z_1^{-1}z_2^{-1})B^3_{d_1,d_2}(z_1,z_2)
=q^{d_2+1}A^0_{d_1+1,d_2+1}(z_1,z_2)+
A^4_{d_1+1,d_2}(q^{-1}z_1,qz_2),
\\
&&(1-q^{d_1-d_2}z_2^{-1})B^4_{d_1,d_2}(z_1,z_2)
=q^{d_2+1}A^1_{d_1,d_2+1}(z_1,z_2)+
A^3_{d_1,d_2}(q^{-1}z_1,qz_2),
\\
&&(1-q^{d_1-d_2}z_2^{-1})B^5_{d_1,d_2}(z_1,z_2)
=z_1^{-1}z_2^{-1}q^{-d_1+1}A^2_{d_1-1,d_2}(z_1,z_2)+
A^0_{d_1,d_2}(q^{-1}z_1,qz_2)
\en
imply the 3-term relation
\begin{multline}\label{bar3}
\varphi_{l_1,l_2,l_1+l_2-k}^{k,k}(z_1,z_2)=
\varphi_{l_1,l_2-1,l_1+l_2-k-1}^{k,k}(z_1,z_2)+\\
z_2^{l_2}\psi_{l_1+l_2-k,k-l_2,l_1}^{k,k}(q^{-1}z_1,qz_2)+\\
(q^{-1}z_1z_2)^{l_1+l_2-k}
\psi_{k-l_2,k-l_1-1,2k-l_1-l_2-1}^{k,k}(z_1,z_2).
\end{multline}
\end{lem}

\begin{prop}\label{lem:4}
The series $A^s_{d_1,d_2}$, $B^s_{d_1,d_2}$
satisfy all relations from Lemmas
$\ref{lem:1}$, $\ref{lem:2}$, $\ref{lem:3}$.
\end{prop}
\begin{proof}
The proposition is proved by a direct calculation.
\end{proof}

{\it Proof of Theorem \ref{g-l}.}
Theorem \ref{g-l} follows from Lemmas \ref{lem:1}--\ref{lem:3},
Proposition \ref{lem:4} and Proposition \ref{unique3}.
\qed

Recall that $\psi_{0,0,0}^{k,k}(z_1,z_2)$ is equal to the character of the
principal subspace $V^k$ of the level $k$ vacuum $\widehat{\mathfrak{sl}_3}$
module.
In this case, one can further sum the six terms in
\eqref{g-l}
to obtain the following corollary:

\begin{cor}\label{CorVI}
\begin{equation}
\label{VI} \ch V^k=\sum_{d_1,d_2\ge 0} z_1^{kd_1} z_2^{kd_2}
q^{k(d_1^2+d_2^2-d_1d_2)} J_{d_1,d_2}(z_1q^{2d_1-d_2},
z_2q^{2d_2-d_1}),
\end{equation}
where
$J_{d_1,d_2}(z_1,z_2)=\bar J_{d_1,d_2}(z_1,z_2) (1-z_1)(1-z_2)(1-z_1z_2)$.
\end{cor}

\begin{cor}
The functions
$I_{d_1,d_2}(z_1,z_2)$
satisfy the following relations:
\begin{equation}
\label{Irec}
I_{d_1,d_2}(z_1,z_2)=\sum_{n_1=0}^{d_1} \sum_{n_2=0}^{d_2}
\frac{z_1^{-n_1}z_2^{-n_2}q^{n_1^2+n_2^2-n_1n_2}}{(q)_{d_1-n_1}(q)_{d_2-n_2}}
I_{n_1,n_2}(z_1,z_2).
\end{equation}
\end{cor}
\begin{proof}
We recall the fermionic formula for the character of $V^k$:
\begin{equation*}
\mathrm{ch} V^k=\sum_{\genfrac{}{}{0pt}{}{n_1,\dots,n_k\ge
0}{m_1,\dots,m_k\ge 0}} \frac{z_1^{\sum in_i}z_2^{\sum im_i}
q^{\sum_{i,j=1}^k \min(i,j) (n_in_j - m_in_j + m_im_j)}}
{(q)_{n_1}\dots (q)_{n_k} (q)_{m_1}\dots (q)_{m_k}}.
\end{equation*}
Summing up all terms with the fixed values of $n_k$ and $m_k$
we obtain the relation
\begin{equation}
\label{Vrec}
\mathrm{ch} V^k=\sum_{n,m\ge 0} \frac{z_1^{kn}z_2^{km}
q^{k(n^2+m^2-mn)}}{(q)_n(q)_m}\mathrm{ch} V^{k-1}(q^{2n-m}z_1,q^{2m-n}z_2).
\end{equation}
Substituting $(\ref{VI})$ into $(\ref{Vrec})$ we obtain
$(\ref{Irec})$.
\end{proof}

\begin{cor}
The function $I_{d_1,d_2}(z_1,z_2)$
is given by the fermionic formula
\be
I_{d_1,d_2}(z_1,z_2)=
\sum_{\{m_i\}_{i>0},\{n_i\}_{i>0}}
\frac{z_1^{-\sum_{i>0} n_i}z_2^{-\sum_{i>0} m_i}
q^{\sum_{i>0}(n_i^2+m_i^2-n_im_i)}}
{(q)_{d_1-n_1}(q)_{n_1-n_2}\dots (q)_{d_2-m_1}(q)_{m_1-m_2}\dots},
\en
where the sum is over all sequences
$\{m_i\}_{i>0},\{n_i\}_{i>0}$ satisfying
$m_i,n_i\in\Z_{\geq0}$ and $d_1\geq n_1\geq n_2\geq\cdots$, $d_2\geq
m_1\geq m_2\geq\cdots,$ and $m_i=n_i=0$ for almost all $i$.
\end{cor}

\begin{prop}\label{prop:I=I}
We have $I_{d_1,d_2}(z_1,z_2)
=\sum_{n=0}^{\min(d_1,d_2)}I_{d_1,d_2,n}(z_1,z_2)$,
where
\bea
&&I_{d_1,d_2,n}(z_1,z_2)
=\frac{1}{(q)_{d_1-n}(q)_{d_2-n}(q)_n}
\label{Iddn}\\
&&\quad
\times\frac{(qz_2)_\infty}
{(qz_1^{-1})_{d_1-n}
(qz_1^{-1}z_2^{-1})_n (q^{d_1-2n+1}z_2)_\infty
(q^{-d_1+2n+1}z_2^{-1})_{d_2-n} (qz_2)_{d_1-n}
 (qz_2^{-1})_n}.
\nonumber
\ena
\end{prop}
\begin{proof}
Using the equality $\psi^{k,k}_{0,0,0}(z_1,z_2)=\ch V^k$ and formula
(\ref{PSIB}) we obtain
\begin{multline}\label{000}
 \ch V^k=\sum_{i\geq0}
\frac{z_2^{ik}q^{i^2k}}{(q)_i(q^{2i}z_2)_\infty(q^{-2i+1}z_2^{-1})_i}
(\chi_B)^k_{0,0}(q^{-i}z_1,q^{2i}z_2)\\
+\sum_{i\geq0} \frac{z_2^{ik}q^{i^2k}}
{(q)_i(q^{2i+1}z_2)_\infty(q^{-2i}z_2^{-1})_{i+1}}
(\chi_B)^{k}_{0,0}(q^iz_1z_2,q^{-2i}z_2^{-1}).
\end{multline}
Recall formula (\ref{chiB}).
Because of the equalities
$$
p^k_{0,0}(m,n,s,z_1,z_2)=z_1^{km}z_2^{kn},\quad
a^k_{0,0}(m,n,s)=k(m^2+n^2-mn),
$$
the right hand side of (\ref{000}) can be rewritten as
$$
\sum_{d_1,d_2\ge 0} z_1^{kd_1} z_2^{kd_2}
q^{k(d_1^2+d_2^2-d_1d_2)} \sum_{n=0}^{\min(d_1,d_2)}
J_{d_1,d_2,n}(z_1q^{2d_1-d_2}, z_2q^{2d_2-d_1}),
$$
where
$J_{d_1,d_2,n}(z_1,z_2)=
\frac{1}{(qz_1)_\infty(qz_2)_\infty(qz_1z_2)_\infty}{I_{d_1,d_2,n}(z_1,z_2)}.$
 Now the
proposition follows from Corollary \ref{CorVI}.
\end{proof}


\section{Whittaker vector
and the character of the vacuum module}

\subsection{The quantum group $\Uv$}
Let $\Uv$ be the quantum group associated to $\mathfrak{sl}_3$.  The
quantum group $\Uv$ is an associative algebra over the field of
rational functions $\C(v)$ in formal variable $v$ with generators
$K^{\pm 1}$, $E_i$, $F_i$, $i=1,2$, satisfying the standard
commutation relations:
\bea
 K_iK_i^{-1}=1,\qquad K_iK_j=K_jK_i,
\hspace{119pt}\notag \\
K_iE_i=v^2 E_iK_i,\qquad K_iF_i=v^{-2}F_iK_i,\qquad K_iE_j=v^{-1} E_jK_i,
\qquad K_iF_j=v F_jK_i,\notag\\
E_iF_i-F_iE_i=\frac{K_i-K_i^{-1}}{v-v^{-1}},\qquad
E_iF_j=F_jE_i,
\hspace{119pt}\notag\\
E_i^2E_j-(v+v^{-1})E_iE_jE_i+E_jE_i^2=0, \qquad
F_i^2F_j-(v+v^{-1})F_iF_jF_i+F_jF_i^2=0\notag
\ena
for all $i,j=1,2$, $i\neq j$.

We extend the algebra $\Uv$ by the operators $K_i^{\pm 1/3}$ in the
obvious way and use the same notation
$\Uv$ for the resulting algebra.

Let $Z\in \Uv$ be the quadratic Casimir operator given by:
\bea
\lefteqn{Z=v^{-2}K_1^{-4/3}K_2^{-2/3}+K_1^{2/3}K_2^{-2/3}+v^2 K_1^{2/3}
K_2^{4/3}
}\notag\\
&&+(v-v^{-1})^2(v^{-1}F_1E_1K_1^{-1/3}K_2^{-2/3}+
vF_2E_2K_1^{2/3}K_2^{1/3}+
v^{-1}F_{13}E_{13}K_1^{-1/3}K_2^{1/3}),\notag
\ena
where $F_{13}=F_2F_1-vF_1F_2$ and $E_{13}=E_1E_2-vE_2E_1$.
\begin{lem}
The element $Z$ is in the center of $\Uv$.
\end{lem}
\begin{proof}
The lemma is proved by a direct calculation.
\end{proof}

We have an anti-isomorphism of $\C(v)$-algebras $\tau$ given by:
\be
\tau:\ \Uv\to\Uv, \quad E_i\mapsto F_i, \quad
F_i\mapsto E_i, \quad K_i\mapsto K_i.
\en
We have an isomorphism of $\C(v)$-algebras $\nu$ given by:
\be
\nu:\ \Uv\to\Uv, \quad E_i\mapsto E_{3-i},
\quad F_i\mapsto F_{3-i}, \quad K_i\mapsto K_{3-i}.
\en

We also consider the quantum group $\Uvv$
with parameter $v^{-1}$.
Denote the generators of $\Uvv$ by
$\bar E_i,\bar F_i,\bar K_i$.

We have an isomorphism of  $\C(v)$-algebras $\sigma$ given by:
\be
\sigma:\ \Uv\to\Uvv, \quad E_i\mapsto \bar E_i,
\quad F_i\mapsto \bar F_i, \quad K_i\mapsto \bar K_i^{-1}.
\en

Clearly, all these maps commute: $\tau\circ\sigma=\sigma\circ\tau$,
$\tau\circ\nu=\nu\circ\tau$, $\nu\circ\sigma=\sigma\circ\nu$.

\subsection{Verma modules}

Let $\la_1, \la_2$ be formal variables. Set $\la_3=-\la_1-\la_2$.

 Let $\C(v,v^{\la_1},v^{\la_2})$ be the field of rational functions in
formal variables $v,v^{\la_1},v^{\la_2}$. Let $\mc R$ be the ring
spanned by all elements of the form $\sqrt{g}$, where
$g\in\C(v,v^{\la_1},v^{\la_2})$ , with the obvious operations of
addition and multiplication. In what follows, we consider the quantum
groups with the extended coefficient ring: $\Uv\otimes_{\C(v)} \mc
R$, $\Uvv\otimes_{\C(v)}\mc R$.  We use the same notation $\Uv$,
$\Uvv$ for the extended algebras.

Let $\V_v$
be the $\Uv$-module generated by the highest weight vector $w$ with the
defining relations:
\be
K_1w=v^{\la_1-\la_2}w, \qquad K_2w=v^{\la_2-\la_3}w, \qquad
E_1w=0, \qquad E_2w=0.
\en

Similarly, let  $\V_{v^{-1}}$ be the $\Uvv$-module
generated by the highest weight vector $\bar w$ with the defining relations:
\be
\bar K_1\bar w=v^{\la_2-\la_1}\bar w, \qquad {\bar K_2}\bar w=
v^{\la_3-\la_2}\bar w,
\qquad \bar E_1w=0, \qquad \bar E_2w=0.
\en

We call $\V_v$ and $\V_{v^{-1}}$ the {\it Verma modules} over $\Uv$
and $\Uvv$, respectively.

We denote by $\V_v(d_1,d_2)\subset \V_v$ and by
$\V_{v^{-1}}(d_1,d_2)\subset \V_{v^{-1}}$ the weight subspaces:
\be
\V_v(d_1,d_2)&=&\{w_1\in \V_v \ |\ K_1w_1=v^{\la_1-\la_2-2d_1+d_2}w_1,\
{K_2}w_1=v^{\la_2-\la_3-2d_2+d_1}w_1\},\\
V_{v^{-1}}(d_1,d_2)&=&\{\bar w_2\in \V_{v^{-1}}\ |\
{\bar K_1}\bar w_2=v^{\la_2-\la_1+2d_1-d_2}\bar w_2,\
{\bar K_2}\bar w_2=v^{\la_3-\la_2+2d_2-d_1}\bar w_2\}.
\en
We have $\V_v=\oplus_{d_1,d_2=0}^\infty \V_v(d_1,d_2)$,
$\V_{v^{-1}}=\oplus_{d_1,d_2=0}^\infty \V_{v^{-1}}(d_1,d_2)$.

\begin{lem} There exists a unique non-degenerate $\mc R$-bilinear pairing
$(\ ,\ ):\  \V_v\otimes \V_{v^{-1}}\to \mc R$ such that $(w,\bar w)=1$ and
\be
(g w_1,\bar w_2)=(w_1,(\sigma\circ\tau) (g)\bar w_2), \qquad
\en
for any $w_1\in\V_v$, $\bar w_2\in \V_{v^{-1}}$, $g\in \Uv$.
\end{lem}
\begin{proof}
The proof is standard.
\end{proof}

In what follows, we use the following notation:
\be
[a]=\frac{v^a-v^{-a}}{v-v^{-1}},
 \qquad [a]!=\prod_{i=1}^a[i], \qquad [a]_b=\prod_{i=0}^{b-1}[a+i].
\en

The Verma module $\V_v$ has the Gelfand-Tsetlin basis
$\{m_{d_1,d_2,n}\ |\ d_1,d_2,n\in\Z_{\geq 0},n\leq\min(d_1,d_2)\}$
(see \cite{J}). In this basis the action of the $\Uv$ is given by:
\bea
K_1m_{d_1,d_2,n}&=&v^{\la_1-\la_2-2d_1+d_2}m_{d_1,d_2,n}, \qquad
{K_2}m_{d_1,d_2,n}=v^{\la_2-\la_3-2d_2+d_1}m_{d_1,d_2,n},\notag \\
E_1m_{d_1,d_2,n}&=&\sqrt{b_1(d_1,d_2,n)}m_{d_1-1,d_2,n-1}+
\sqrt{b_2(d_1,d_2,n)}m_{d_1-1,d_2,n},\notag\\
F_1m_{d_1,d_2,n}&=&\sqrt{b_1(d_1+1,d_2,n+1)}m_{d_1+1,d_2,n+1}+
\sqrt{b_2(d_1+1,d_2,n)}m_{d_1+1,d_2,n},\notag\\
E_2m_{d_1,d_2,n}&=&\sqrt{a(d_1,d_2,n)}m_{d_1,d_2-1,n}, \quad
F_2m_{d_1,d_2,n}=\sqrt{a(d_1,d_2+1,n)}m_{d_1,d_2+1,n},\notag
\ena
where
\bea
a(d_1,d_2,n)&=&[d_2-n][\la_2-\la_3+d_1-d_2-n+1], \notag\\
b_1(d_1,d_2,n)&=&\frac{[d_2-n+1][n][\la_2-\la_3-n+1]
[\la_1-\la_3-n+2]}{[\la_2-\la_3+d_1-2n+1][\la_2-\la_3+d_1-2n+2]},\notag \\
b_2(d_1,d_2,n)&=&\frac{[d_1-n][\la_2-\la_3+d_1-d_2-n]
[\la_2-\la_3+d_1-n+1]}{[\la_2-\la_3+d_1-2n][\la_2-\la_3+d_1-2n+1]}\notag \\
&&\hspace{200pt}\times \frac{[\la_1-\la_2-d_1+n+1]}{1}. \notag
\ena

\begin{rem} Our formulas are identified with the formulas in \cite{J}
as follows. Let the vector $m(d_1,d_2,n)$ correspond
to the Gelfand-Tsetlin pattern in \cite{J} given by:
\be
\left(\begin{matrix}
-\la_3 &&-\la_2&&-\la_1 \\
&-\la_3-n&&-\la_2-d_1+n&\\
&&-\la_3-d_2 &&
\end{matrix}\right)
\en
Then the action of $g\in \Uv$ in our paper is given by
formulas for the action of $\nu(g)$ in \cite{J}.
\end{rem}

Similarly, we have a Gelfand-Tsetlin basis of the Verma module $\V_{v^{-1}}$,
\newline $\{\bar m_{d_1,d_2,n}\ |\ d_1,d_2,n\in\Z_{\geq
0},n\leq\min(d_1,d_2)\}$.

\begin{lem} We have
\be
(m_{d_1,d_2,n},\bar m_{d_1',d_2',n'})=
\delta_{d_1,d_1'}\delta_{d_2,d_2'}\delta_{n,n'}.
\en
\end{lem}
\begin{proof}
The Shapovalov form on $\V_v$ is the unique
non-degenerate symmetric bilinear form such
that the length of the highest weight vector $w$ is 1 and every
$g\in\Uv$ is dual to $\tau(g)$.  According to \cite{J}, the
Gelfand-Tsetlin basis is orthonormal with respect to the Shapovalov
form in $\V_v$. The lemma follows.
\end{proof}

\subsection{Whittaker vectors}
We call a series $\omega=\sum_{d_1,d_2=0}^\infty \omega_{d_1,d_2}$,
$\omega_{d_1,d_2}\in \V_v(d_1,d_2)$, the {\it Whittaker vector} 
if $\omega_{0,0}=1$ and
\be
E_1K_1^{-1}\omega_{d_1,d_2}=\frac{1}{1-v^2}\omega_{d_1-1,d_2}, \qquad
E_2\omega_{d_1,d_2}=\frac{1}{1-v^2}\omega_{d_1,d_2-1}\notag.
\en

We call a series $\bar \omega=\sum_{d_1,d_2=0}^\infty \bar \omega_{d_1,d_2}$,
$\bar \omega_{d_1,d_2}\in \V_{v^{-1}}(d_1,d_2)$, the
{\it dual Whittaker vector} if
$\bar \omega_{0,0}=1$ and
\be
\bar E_1\bar \omega_{d_1,d_2}=\frac{v}{1-v^{-2}}\bar \omega_{d_1-1,d_2}, \qquad
\bar E_2\bar K_2\bar \omega_{d_1,d_2}=\frac{v}{1-v^{-2}}\bar \omega_{d_1,d_2-1}.\notag
\en

\begin{rem}
Because of the quantum Serre relations, there are no non-zero vectors
invariant under the action of $E_1$ and $E_2$ (except for the
multiples of the highest weight vector $w$). The operators $e_1=E_1K_1^{-1}$
and $e_2=E_2$ satisfy
\be
e_1^2e_2-(1+v^2)e_1e_2e_1+v^2 e_2e_1^2=0, \qquad
e_2^2e_1-(1+v^{-2})e_2e_1e_2+v^{-2} e_1e_2^2=0,
\en
which does not prohibit the existence of non-trivial Whittaker vectors.
\end{rem}

Let $r(d_1,d_2,n)$ and $s(d_1,d_2)$ be given by the formulas:
\bea
r(d_1,d_2,n)&=&(\la_3-\la_2-d_1-1)n+n^2+(\la_2-\la_1+1)d_1+d_1^2,\notag \\
s(d_1,d_2)&=&-d_1^2-d_2^2+d_1d_2+(\la_1-\la_2)d_1+(\la_2-\la_3)d_2.\notag
\ena

Let $c(d_1,d_2,n)$ be given by the formula:
\bea
&&c(d_1,d_2,n)=\frac{1}{[d_1-n]![d_2-n]![n]!} \notag\\
&&\times  \frac{[\la_2-\la_3+2]_\infty}
{[\la_1-\la_2-d_1+n+1]_{d_1-n}[\la_1-\la_3-n+2]_n[\la_2-\la_3+2]_{d_1-n}
}\notag \\
&&\times  \frac{1}{[\la_2-\la_3+d_1-d_2-n+1]_{d_2-n}
[\la_2-\la_3+d_1-2n+2]_{\infty}[\la_2-\la_3-n+1]_n}.\notag
\ena
Note that  $c(d_1,d_2,n)$ is a rational function in
$v,v^{\la_1},v^{\la_2}$, which is invariant under the change
$v\mapsto v^{-1}$, $v^{\la_1}\mapsto v^{-\la_1}$,
$v^{\la_2}\mapsto v^{-\la_2}$.

Set
\bea
 \omega_{d_1,d_2,n}&=&\frac{1}{(1-v^2)^{d_1+d_2}}v^{r(d_1,d_2,n)}
\sqrt{c(d_1,d_2,n)} \ m_{d_1,d_2,n},\notag\\
\bar \omega_{d_1,d_2,n}&=&\frac{1}{(1-v^{-2})^{d_1+d_2}}
v^{-r(d_1,d_2,n)+s(d_1,d_2)}\sqrt{c(d_1,d_2,n)}\ \bar m_{d_1,d_2,n}.\notag
\ena

\begin{thm}\label{whit vector}
The Whittaker vector $\omega$ and the dual Whittaker vector $\bar \omega$ exist,
are unique and are given by the formula:
\be
\omega_{d_1,d_2}=\sum_{n=0}^{\min(d_1,d_2)} \omega_{d_1,d_2,n},\qquad
\bar \omega_{d_1,d_2}=\sum_{n=0}^{\min(d_1,d_2)}\bar \omega_{d_1,d_2,n}.
\en
\end{thm}
\begin{proof}
The theorem is proved by a direct calculation.
\end{proof}

\subsection{Toda recursion}
In this section we describe a relation between the
Whittaker vectors, the Toda recursion and the characters of
$\hat{\mathfrak{n}}$-modules.
 Such relations hold with the following identification:
\be
q=v^2, \qquad z_1=q^{\la_1-\la_2+1}, \qquad z_2=q^{\la_2-\la_3+1},
\en
which we always assume in this section.

The following lemma can be extracted from \cite{E}.

\begin{lem}\label{w-t}
The functions $(\omega_{d_1,d_2},\bar \omega_{d_1,d_2})$ satisfy the
Toda recursion given in Proposition \ref{prop:Toda}.
\end{lem}
\begin{proof}
The Casimir operator $Z\in\Uv$ acts in the cyclic module $\V_v$
by a constant which is readily computed on the highest weight vector.
Therefore, we have
\be
(Z\omega_{d_1,d_2},\bar \omega_{d_1,d_2})=
(q^{-\la_1-1}+q^{-\la_2}+q^{-\la_3+1})(\omega_{d_1,d_2},\bar \omega_{d_1,d_2}).
\en
On the other hand,
\be
(Z\omega_{d_1,d_2},\bar \omega_{d_1,d_2})=
((v^{-2}K_1^{-4/3}K_2^{-2/3}+K_1^{2/3}K_2^{-2/3}+v^2 K_1^{2/3}K_2^{4/3}
)\omega_{d_1,d_2},\bar \omega_{d_1,d_2})\\
+(v-v^{-1})^2\left(v^{-1}
(E_1K_1^{-1/3}K_2^{-2/3}\omega_{d_1,d_2},\bar E_1\bar \omega_{d_1,d_2})
+v(E_2K_1^{2/3}K_2^{1/3}\omega_{d_1,d_2},\bar E_2\bar \omega_{d_1,d_2})\right)\\
=(q^{-\la_1+d_1-1}+q^{-\la_2-d_1+d_2}+q^{-\la_3-d_2+1})
(\omega_{d_1,d_2},\bar \omega_{d_1,d_2})\\ -q^{-\la_2-d_1+d_2}
(\omega_{d_1-1,d_2},\bar \omega_{d_1-1,d_2})-
q^{-\la_3-d_2+1}(\omega_{d_1,d_2-1},\bar \omega_{d_1,d_2-1}).
\en
In this computation we used $\bar E_{13}\bar \omega_{d_1,d_2}=0$.

The lemma follows.
\end{proof}

\begin{cor}
We have $I_{d_1,d_2}(z_1,z_2)=(\omega_{d_1,d_2},\bar \omega_{d_1,d_2})$,
where
$I_{d_1,d_2}(z_1,z_2)$ is given by \eqref{Idd}.
\end{cor}
\begin{proof}
The corollary follows from the uniqueness of the solution of the Toda
recursion, Lemma \ref{w-t} and \eqref{Idd}.
\end{proof}

Recall the functions $I_{d_1,d_2,n}(z_1,z_2)$ given by \eqref{Iddn}.
The following theorem explains the meaning of these functions from the
point of view of Whittaker vectors.

\begin{thm}\label{terms}
We have
\be
I_{d_1,d_2,n}(z_1,z_2)=(\omega_{d_1,d_2,n},\bar \omega_{d_1,d_2,n}).
\en
\end{thm}
\begin{proof}
It is easy to see from the explicit formulas that there exist
integers $r(d_1,d_2,n)$ such that
\be
v^{r(d_1,d_2,n)}\sqrt{I_{d_1,d_2,n}(z_1,z_2)}\ m_{d_1,d_2,n}=\omega_{d_1,d_2,n}.
\en
Let $\bar I_{d_1,d_2,n}(z_1,z_2)$ be obtained from
$I_{d_1,d_2,n}(z_1,z_2)$ by the change
$v\mapsto v^{-1}$, $v^{\la_1}\mapsto v^{-\la_1}$,
$v^{\la_2}\mapsto v^{-\la_2}$. Then we have
\be
v^{-r(d_1,d_2,n)+s(d_1,d_2)}\sqrt{\bar I_{d_1,d_2,n}(z_1,z_2)}\
\bar m_{d_1,d_2,n}=\bar \omega_{d_1,d_2,n}.
\en
It is also easy to check explicitly that
\be
v^{s(d_1,d_2)}\sqrt{\bar I_{d_1,d_2,n}(z_1,z_2)}=
\sqrt{I_{d_1,d_2,n}(z_1,z_2)}.
\en
The theorem follows.
\end{proof}

\begin{rem}
In fact, we guessed the formulas for the Whittaker vectors in
Theorem \ref{whit vector}  expecting that Theorem \ref{terms} is true.
\end{rem}

We recover the result of Proposition \ref{prop:I=I}.
\begin{cor}
We have $I_{d_1,d_2}(z_1,z_2)=
\sum_{n=0}^{\min(d_1,d_2)} I_{d_1,d_2,n}(z_1,z_2)$.
\end{cor}

\begin{rem}
In \cite{IS} the classical limit (corresponding to $v\to 1$) of the
function $I_{d_1,d_2}(z_1,z_2)$ is written as a sum of $\min(d_1,d_2)$
factorized terms. The classical limits of our terms differ from the
terms in \cite{IS}.
\end{rem}


\noindent
{\it Acknowledgments.}\quad
Research of BF is partially supported by RFBR Grants 04-01-00303 and
05-01-01007, INTAS 03-51-3350, NSh-2044.2003.2 and RFBR-JSPS Grant 05-01-02934YaFa.
Research of EF is partially supported by the
RFBR Grants 06-01-00037, 07-02-00799.
Research of MJ is supported by
the Grant-in-Aid for Scientific Research B--18340035.
Research of TM is supported by
the Grant-in-Aid for Scientific Research B--17340038.
Research of EM is supported by NSF grant DMS-0601005.

Most of the present work has been carried out during the visits
of BF, EF and EM to Kyoto University,
they wish to thank the University for hospitality.
MJ is grateful to J. Shiraishi and T. Oda for discussions
and information concerning Whittaker functions.
Last but not least, the authors
are indebted to Tambara Institute for Mathematical Sciences,
the University of Tokyo, for offering
an excellent opportunity for concentration in a beautiful
environment.

\end{document}